\documentclass[12pt,letterpaper]{amsart}

\allowdisplaybreaks
\usepackage{mathpazo}
\usepackage{amsmath,amsfonts,amssymb}
\usepackage{amsrefs}
\usepackage{amsthm}
\usepackage{latexsym,amsmath,amssymb,amsfonts}
\usepackage{rotating}
\usepackage{mathrsfs}
\usepackage{xypic} \xyoption{all}
\usepackage{amscd}
\usepackage{hyperref} 
\usepackage{euscript}
\usepackage{hhline}
\usepackage{graphicx,epstopdf}
\usepackage{epsfig}
\usepackage{xcolor}
\usepackage{textcomp}
\usepackage[all,color]{xy}

\newlength{\fighskip} \fighskip=2pt
\newlength{\figvskip} \figvskip=3pt

\usepackage{hyperref}
\newcommand*{\figbox}[2]{{
  \def\figscale{#1}
  \def\arraystretch{0.8}
  \arraycolsep=0pt
  \begin{array}{c}
    \vbox{\vskip\figscale\figvskip
      \hbox{\hskip\figscale\fighskip
        \includegraphics[scale=\figscale]{#2}}}
  \end{array}}}

\advance\textheight by \topskip
\numberwithin{equation}{section}

\newcommand{\C}{\mathbb{C}}

\newcommand{\R}{\mathbb{R}}

\newcommand{\Z}{\mathbb{Z}}

\newcommand{\E}{{\mathcal E}}

\newcommand{\pa}{\partial}

\newcommand{\W}{\mathcal{W}}

\DeclareMathOperator{\Sym}{Sym}
\DeclareMathOperator{\Hom}{Hom}

\DeclareMathOperator{\Tr}{Tr}

\newcommand{\A}{\mathcal A}

\theoremstyle{plain}
\newtheorem{thm}{Theorem}[section]
\newtheorem{thm-defn}{Theorem/Definition}[section]
\newtheorem{lem}[thm]{Lemma}
\newtheorem{lem-defn}[thm]{Lemma/Definition}
\newtheorem{prop}[thm]{Proposition}
\newtheorem{cor}[thm]{Corollary}

\theoremstyle{definition}
\newtheorem{defn}[thm]{Definition}
\newtheorem{notn}[thm]{Notation}
\newtheorem{eg}[thm]{Example}

\theoremstyle{remark}
\newtheorem{rmk}{Remark}[section]

\allowdisplaybreaks[4]  

\begin{document}
\title[Kapranov, Fedosov and 1-loop exact BV on K\"ahler manifolds]{Kapranov's $L_\infty$ structures, Fedosov's star products, and one-loop exact BV quantizations on K\"ahler manifolds}

\author[Chan]{Kwokwai Chan}
\address{Department of Mathematics\\ The Chinese University of Hong Kong\\ Shatin\\ Hong Kong}
\email{kwchan@math.cuhk.edu.hk}

\author[Leung]{Naichung Conan Leung}
\address{The Institute of Mathematical Sciences and Department of Mathematics\\ The Chinese University of Hong Kong\\ Shatin \\ Hong Kong}
\email{leung@math.cuhk.edu.hk}

\author[Li]{Qin Li}
\address{Shenzhen Institute for Quantum Science and Engineering\\ Southern University of Science and Technology\\ Shenzhen\\China}
\email{liqin@sustech.edu.cn}

\thanks{}

\begin{abstract}
We study quantization schemes on a K\"ahler manifold and relate several interesting structures.
We first construct Fedosov's star products on a K\"ahler manifold $X$ as quantizations of Kapranov's $L_\infty$-algebra structure.
Then we investigate the Batalin-Vilkovisky (BV) quantizations associated to these star products. 
A remarkable feature is that they are all one-loop exact, meaning that the Feynman weights associated to graphs with two or more loops all vanish.
This leads to a succinct cochain level formula in de Rham cohomology for the algebraic index.
\end{abstract}

\maketitle


\section{Introduction}

K\"ahler manifolds possess very rich structures because they lie at the crossroad of complex geometry and symplectic geometry. On the other hand, as symplectic manifolds which admit natural polarizations (namely, the complex polarization), K\"ahler manifolds provide a natural ground for constructing quantum theories. In particular, there have been extensive studies on deformation quantization on K\"ahler manifolds. A notable example is the Berezin-Toeplitz quantization, which is closely related to geometric quantization in the complex polarization \cite{Bordemann-Meinrenken, Bordemann, Karabegov96, Karabegov00, Karabegov07, Karabegov, Ma-Ma, Neumaier}.

This paper is another attempt to understand the relation between the K\"ahlerian condition and properties of quantum theories.
Our starting point is Kapranov's famous construction of an {\em $L_\infty$-algebra structure} for a K\"ahler manifold in \cite{Kapranov}, which was motivated by the study of Rozansky-Witten theory \cite{RW} (and also \cite{Kontsevich2}) and has since been playing important roles in many different subjects.

Let $X$ be a K\"ahler manifold. Using the Atiyah class of the holomorphic tangent bundle $TX$, Kapranov constructed a natural $L_\infty$-algebra structure on the Dolbeault complex $\A_X^{0,\bullet - 1}(TX)$, enabling us to view $TX[-1]$ as a Lie algebra object in the derived category of coherent sheaves on $X$. 
In Section \ref{section: L-infty-structure}, we reformulate this structure as a flat connection $D_K$ (where the subscript ``K'' stands for ``Kapranov'') on the holomorphic Weyl bundle $\W_X$ over $X$. 
On the other hand, {\em Fedosov abelian connections}, which give rise to {\em deformation quantizations} or {\em star products} on $X$, are connections of the form $D = \nabla-\delta+\frac{1}{\hbar}[I,-]_{\star}$ on the complexified Weyl bundle $\W_{X,\mathbb{C}}$  satisfying $D^2=0$; here $\nabla$ is the Levi-Civita connection, $I\in\A^1(X,\W_{X,\C})$ is a $1$-form valued section of $\W_{X,\C}$, and $\star$ is the fiberwise Wick product on $\W_{X,\mathbb{C}}$. The flatness condition $D^2 = 0$ is equivalent to the {\em Fedosov equation}
\begin{equation}\label{eqn:Fedosov-equation-Wick-intro}
\nabla I-\delta I+\frac{1}{\hbar}I\star I+ R_\nabla=\alpha.
\end{equation}
Our first main result says that Kapranov's $L_\infty$ structure can naturally be quantized to produce Fedosov abelian connections $D_{F,\alpha}$ (where the subscript ``F'' stands for ``Fedosov''):
\begin{thm}[=Theorem \ref{proposition: Fedosov-connection-general}]
	For a representative $\alpha$ of any given formal cohomology class $[\alpha] \in \hbar H^2_{dR}(X)[[\hbar]]$ of type $(1,1)$, there exists a Fedosov abelian connection $D_{F,\alpha} = \nabla-\delta+\frac{1}{\hbar}[I_\alpha,-]_{\star}$
	such that $D_{F,\alpha}|_{\W_X} = D_K$.
\end{thm}
There are some interesting features of the resulting star products $\star_\alpha$ on $X$.
First of all, they are of so-called {\em Wick type}, which roughly means that the corresponding bidifferential operators respect the complex polarization (Proposition \ref{proposition: Fedosov-quantization-1-1-class-Wick-type}). By computing the Karabegov form associated to our star products, we can also show that every Wick type star product arises from our construction (Corollary \ref{corollary:Wick-type}).
More importantly, our solutions of the Fedosov equation \eqref{eqn:Fedosov-equation-Wick-intro} satisfy a gaugue condition (Proposition \ref{proposition-fedosov-gauge-conditions}) which is different from all previous constructions of Fedosov quantization on K\"ahler manifolds. Because of this, our construction is more consistent with the Berezin-Toeplitz quantization\footnote{Combining with the results in \cite{CLL-PartI}, we can show that these Fedosov star products actually coincide with the Berezin-Toeplitz star products \cite{CLL-PartIII}.} and also the local picture that $z$ acts as the {\em creation operator} (which is classical) while $\bar{z}$ acts as the {\em annihilation operator} $\hbar\frac{\partial}{\partial z}$ (which is quantum).
Furthremore, our Fedosov quantization is, in a certain sense, {\em polarized} because only half of the functions, namely, the anti-holomorphic ones, receive quantum corrections.
See Section \ref{section: Fedosov-quantization} for more details.

Next we go from deformation quantization to quantum field theory (QFT).
From the QFT viewpoint, the Fedosov quantization describes the local data of a quantum mechanical system, namely, the cochain complex 
$$
(\A_X^\bullet(\mathcal{W}_{X,\C}), D_{F,\alpha})
$$
gives the {\em cochain complex of local quantum observables} of a sigma model from $S^1$ to the target manifold $X$.
To obtain {\em global quantum observables} and also define the {\em correlation functions}, we construct the {\em Batalin-Vilkovisky (BV) quantization} \cite{BV} of this quantum mechanical system, which comes with a map from local to global quantum observables called the {\em factorization map}.
We will mainly follow Costello's approach to the BV formalism \cite{Kevin-book} and rely on Costello-Gwilliam's foundational work on factorization algebras in QFT \cite{Kevin-Owen, Kevin-Owen-2}.

To construct a BV quantization, one wants to produce solutions of the quantum master equation (QME) (see Lemma \ref{lemma: BV-operator-differential}, Definition \ref{definition: QME} and the equation \eqref{equation: quantum-master-equation}):
\begin{equation}\label{equation: quantum-master-equation-intro}
Q_{BV}(e^{r/\hbar})=0,
\end{equation}
where $Q_{BV}:=\nabla+\hbar\Delta+\frac{1}{\hbar}d_{TX}R_{\nabla}$ is the so-called {\em BV differential}, by running the {\em homotopy group flow} which is defined by choosing a suitable propagator. For general symplectic manifolds, it was shown in the joint work \cite{GLL} of the third author with Grady and Si Li that canonical solutions of the QME can be constructed by applying the homotopy group flow operator to solutions of the Fedosov equation \eqref{eqn:Fedosov-equation-Wick-intro}.

In the K\"ahler setting, it is desirable to choose the propagator differently to make it more compatible with the complex polarization. This leads to what we call the {\em polarized propagator} (see Definition \ref{definition: propagator-polarized}) and hence a slightly different form of the canonical solutions to the QME. More precisely, in Theorem \ref{theorem: Fedosov-connection-RG-flow-QME}, we will construct the canonical solution $e^{\tilde{R}_\nabla/2\hbar}\cdot e^{\gamma_\infty/\hbar}$ of the QME \eqref{equation: quantum-master-equation-intro} from a solution $\gamma$ of the Fedosov equation \eqref{equation: Fedosov-equation-gamma} (which is equivalent to \eqref{eqn:Fedosov-equation-Wick-intro}); here the curvature term $\tilde{R}_\nabla$ appears precisely because of the different choice of the propagator.
The second main result of this paper is that the resulting BV quantization is {\em one-loop exact}, meaning that the above canonical solution of the QME admits a Feynman graph expansion that involves {\em only trees and one-loop graphs}:
\begin{thm}[=Theorem \ref{theorem: gamma-infty-one-loop}]\label{theorem: gamma-infty-one-loop-intro}
	Let $\gamma$ be a solution of the Fedosov equation \eqref{equation: Fedosov-equation-gamma}.
	Then the Feynman weight associated to a graph $\mathcal{G}$ with two or more loops vanishes, i.e.,
	$$W_{\mathcal{G}}(P, d_{TX}\gamma) = 0\quad \text{ whenever $b_1(\mathcal{G}) \geq 2$}.$$
	Hence, the graph expansion of the canonical solution of the QME associated to the Fedosov connection $D_{F,\alpha}$ by Theorem \ref{theorem: Fedosov-connection-RG-flow-QME} involves only trees and one-loop graphs:
	$$
	\gamma_\infty=\sum_{\mathcal{G}:\ \text{connected},\ b_1(\mathcal{G})=0,1}\frac{\hbar^{g(\mathcal{G})}}{|\text{Aut}(\mathcal{G})|}W_{\mathcal{G}}(P, d_{TX}\gamma). 
	$$
\end{thm}
 This is in sharp contrast with the general symplectic case studied in \cite{GLL} where the BV quantization involves quantum corrections from graphs with any number of loops.
The same kind of one-loop exactness has been observed in a few cases before, including the holomorphic Chern-Simons theory studied by Costello \cite{Kevin-CS} and a sigma model from $S^1$ to the target $T^*Y$ (cotangent bundle of a smooth manifold $Y$) studied by Gwilliam-Grady \cite{Gwilliam-Grady}. Theorem \ref{theorem: gamma-infty-one-loop-intro} shows that K\"ahler manifolds provide a natural geometric ground for producing such one-loop exact QFTs.

\begin{rmk}
	If the Feynman weights associated to graphs of higher genera ($\geq 2$) give rise to {\em exact} differential forms and thereby contributing trivially to the computation of correlation functions, we may call the quantization {\em cohomologically one-loop exact}. This is much more commonly found in the mathematical physics literature and should be distinguished from our notion of one-loop exactness here.
\end{rmk}

As in the symplectic case \cite{GLL}, from the canonical QME solution $e^{\tilde{R}_\nabla/2\hbar}\cdot e^{\gamma_\infty/\hbar}$, we obtain the local-to-global factorization map, which can be used to define the correlation function $\langle f\rangle$ of a smooth function $f \in C^\infty(X)[[\hbar]]$ (see Definition \ref{definition:correlation-functions} and Proposition \ref{proposition: leading-term-correlation-function}). The association $\Tr: f \mapsto \langle f\rangle$ then gives a {\em trace} of the star product $\star_\alpha$ associated to the Fedosov connection $D_{F, \alpha}$ (Corollary \ref{corollary:trace}).


As an application of the one-loop exactness of our BV quantization in Theorem \ref{theorem: gamma-infty-one-loop-intro}, we deduce a novel {\em cochain level} formula for the for the {\em algebraic index} $\Tr(1)$, which is the correlation function of the constant function $1$, and thus a particularly neat presentation of the {\em algebraic index theorem}:
\begin{thm}[=Theorem \ref{theorem:algebraic-index-theorem} \& Corollary \ref{corollary:algebraic-index-theorem}]
	We have
	$$\sigma\left(e^{\hbar\iota_{\Pi}} (e^{\tilde{R}_\nabla/2\hbar}e^{\gamma_\infty/\hbar}) \right)
	= \hat{A}(X)\cdot e^{-\frac{\omega_\hbar}{\hbar}+\frac{1}{2}\Tr(\mathcal{R}^+)} = \text{Td}(X)\cdot e^{-\frac{\omega_\hbar}{\hbar}+\Tr(\mathcal{R}^+)},$$
	where $Td(X)$ is the Todd class of $X$ and $\mathcal{R}^+$ is the curvature of the holomorphic tangent bundle defined in \eqref{equation: R-plus}. In particular, we obtain the {\em algebraic index theorem}, namely, the trace of the function $1$ is given by 
	\begin{align*}
	\Tr(1) = \int_X\hat{A}(X)\cdot e^{-\frac{\omega_\hbar}{\hbar}+\frac{1}{2}\Tr(\mathcal{R}^+)}
	=  \int_X\text{Td}(X)\cdot e^{-\frac{\omega_\hbar}{\hbar}+\Tr(\mathcal{R}^+)}.
	\end{align*}
\end{thm}
This theorem can be regarded as a cochain level enhancement of the result in \cite{GLL}.
In the forthcoming work \cite{CLL-PartIII}, by combining with the results in \cite{CLL-PartI}, we will prove that the star product $\star_\alpha$ associated to $\alpha = \hbar\cdot\Tr(\mathcal{R}^+)$ is precisely the {\em Berezin-Toeplitz star product} on a prequantizable K\"ahler manifold studied in \cite{Bordemann-Meinrenken, Bordemann, Karabegov}. In this case, the algebraic index theorem is of the simple form:
$$
\Tr(1)=\int_X\text{Td}(X)\cdot e^{\omega/\hbar}.
$$


\subsection*{Acknowledgement}
\

We thank Si Li and Siye Wu for useful discussions. The first named author thanks Martin Schlichenmaier and Siye Wu for inviting him to attend the conference GEOQUANT 2019 held in September 2019 in Taiwan, in which he had stimulating and very helpful discussions with both of them as well as Jørgen Ellengard Andersen, Motohico Mulase, Georgiy Sharygin and Steve Zelditch.

K. Chan was supported by a grant of the Hong Kong Research Grants Council (Project No. CUHK14303019) and direct grant (No. 4053395) from CUHK.
N. C. Leung was supported by grants of the Hong Kong Research Grants Council (Project No. CUHK14301117 \& CUHK14303518) and direct grant (No. 4053400) from CUHK.
Q. Li was supported by a grant from National Natural Science Foundation of China  (Project No. 12071204), and Guangdong Basic and Applied Basic Research Foundation (Project No. 2020A1515011220).

\section{From Kapranov to Fedosov}

There have been extensive studies on the Fedosov quantization of K\"ahler manifolds. In this section, we give a new construction of the Fedosov quantization (i.e., solutions of the Fedosov equation for abelian connections on the Weyl bundle) as a natural quantization of Kapranov's $L_\infty$-algebra structure on a K\"ahler manifold.


The organization of this section is as follows: In Section \ref{section: kalher-geom}, we review some basic K\"ahler geometry, including the geometry of Weyl bundles on a K\"ahler manifold $X$. In Section \ref{section: L-infty-structure}, we recall the  $L_\infty$-algebra structure introduced by Kapranov \cite{Kapranov}, reformulated using the geometry of Weyl bundles on $X$. 
In Section \ref{section: Fedosov-quantization}, we construct Fedosov's flat connections as quantizations of Kapranov's $L_\infty$-algebra structure, which produce star products on $X$ of Wick type. We also compute the Karabegov forms associated to such star products and prove that every Wick type star product arises from our construction.

\subsection{Preliminaries in K\"ahler geometry}\label{section: kalher-geom}

\subsubsection{Some basic identities}
\

We first collect some basic identities in K\"ahler geometry, which are needed in later computations.

First of all, writing the K\"ahler form as
$$
\omega=\omega_{\alpha\bar{\beta}}dz^\alpha\wedge d\bar{z}^\beta=\sqrt{-1}g_{\alpha\bar{\beta}}dz^\alpha\wedge d\bar{z}^\beta,
$$
where we adopt the convention that $\omega^{\bar{\gamma}\alpha}\omega_{\alpha\bar{\beta}}=\delta_{\bar{\beta}}^{\bar{\gamma}}$, then a simple computation shows that $g^{\alpha\bar{\beta}}=-\sqrt{-1}\omega^{\alpha\bar{\beta}}$.

In local coordinates, the curvature of the Levi-Civita connection is given by 
$$
\nabla^2(\partial_{x^k})=R_{ijk}^ldx^i\wedge dx^j\otimes\partial_{x^l},
$$
or in complex coordinates:
$$
\nabla^2(\partial_{z^k})=R_{i\bar{j}k}^l dz^i\wedge d\bar{z}^j\otimes\partial_{z^l},
$$
where the coefficients can be written as the derivatives of Christoffel symbols:
$$
R_{i\bar{j}k}^l=-\partial_{\bar{z}^j}(\Gamma_{ik}^l).
$$

We define $R_{i\bar{j}k\bar{l}}$ by 
\begin{align*}
R_{i\bar{j}k\bar{l}}:=g(\mathcal{R}(\partial_{z^i},\partial_{\bar{z}^j})\partial_{z^k},\partial_{\bar{z}^l})
=g(R_{i\bar{j}k}^m\partial_{z^m},\partial_{\bar{z}^l})
=R_{i\bar{j}k}^m g_{m\bar{l}}.
\end{align*}
We can also compute the curvature on the cotangent bundle:
\begin{align*}
\mathcal{R}(\partial_{z^i},\partial_{\bar{z}^j})(y^k) & = -\mathcal{R}(\partial_{\bar{z}^j},\partial_{z^i})(y^k) = -\partial_{\bar{z}^j}(-\Gamma_{il}^kdz^i\otimes y^l)\\
& = \partial_{\bar{z}^j}(\Gamma_{ij}^k)d\bar{z}^j\wedge dz^i\otimes y^l = R_{i\bar{j}k}^ldz^i\wedge d\bar{z}^j\otimes y^l. 
\end{align*}
The following computation shows that the curvature operator $\mathcal{R}$ can be expressed as a bracket:
\begin{align*}
\frac{\sqrt{-1}}{\hbar}[R_{i\bar{j}k\bar{l}}dz^i\wedge d\bar{z}^j\otimes y^k\bar{y}^l,y^m]_\star
& = -\frac{\sqrt{-1}}{\hbar}R_{i\bar{j}k\bar{l}}dz^i\wedge d\bar{z}^j\otimes\left(y^m\star y^k\bar{y}^l-y^my^k\bar{y}^l\right)\\
& = -\sqrt{-1}R_{i\bar{j}k\bar{l}}\omega^{m\bar{l}}dz^i\wedge d\bar{z}^j\otimes y^k\\
& = -\sqrt{-1}R_{i\bar{j}k}^ng_{n\bar{l}}\omega^{m\bar{l}}dz^i\wedge d\bar{z}^j\otimes y^k\\
& = -\sqrt{-1}R_{i\bar{j}k}^n(-\sqrt{-1}\omega_{n\bar{l}})\omega^{m\bar{l}}dz^i\wedge d\bar{z}^j\otimes y^k\\
& = R_{i\bar{j}k}^mdz^i\wedge d\bar{z}^j\otimes y^k\\
& = \nabla^2(y^m). 
\end{align*}

For later computations, we also use the notation $\mathcal{R}^+$ to denote the curvature of the holomorphic tangent bundle:
\begin{equation}\label{equation: R-plus}
\mathcal{R}^+:=R_{i\bar{j}k}^mdz^i\wedge d\bar{z}^j\otimes(\partial_{z^m}\otimes y^k). 
\end{equation}
In particular, we have an explicit formula for its trace:
\begin{equation}\label{equation: trace-R-plus}
\Tr(\mathcal{R}^+)=R_{i\bar{j}k}^k=R_{i\bar{j}k\bar{l}}g^{k\bar{l}}=-\sqrt{-1}R_{i\bar{j}k\bar{l}}\omega^{k\bar{l}}.
\end{equation}

\subsubsection{Weyl bundles on K\"ahler manifolds}
\

Here we recall the definitions and basic properties of various types of Weyl bundles on symplectic and K\"ahler manifolds. 
\begin{defn}\label{definition: real-Weyl-bundle}
For a symplectic manifold $(M,\omega)$, its {\em (real) Weyl bundle} is defined as
$$\W_{M,\R} := \widehat{\Sym}(T^*M_{\mathbb{R}})[[\hbar]],$$
where $\widehat{\Sym}(T^*M_{\mathbb{R}})$ is the completed symmetric power of the cotangent bundle $T^*M_{\mathbb{R}}$ of $M$ and $\hbar$ is a formal variable.
A (smooth) section $a$ of this infinite rank bundle is given locally by a formal series
\[
a(x,y)=\sum_{k, l\geq 0} \sum_{i_1,\dots,i_l \geq0} \hbar^k a_{k,i_1\cdots i_l}(x)y^{i_1}\cdots y^{i_l},
\]
where the $a_{k,i_1\cdots i_l}(x)'s$ are smooth functions on $M$.
\end{defn}
\begin{rmk}\label{remark: symmetric-tensor-product}
  We use the following notation for the symmetric tensor power:
  $$
  y^{i_1}\cdots y^{i_l}:=\sum_{\tau\in S_l}y^{i_{\tau(1)}}\otimes\cdots\otimes y^{i_{\tau(l)}}.
  $$
  Here the product on the tensor algebra is given by
  $$
  (y^{i_1}\otimes\cdots\otimes y^{i_k})\cdot (y^{i_{k+1}}\otimes\cdots\otimes y^{i_{k+l}}):=\sum_{\tau\in\text{Sh}(k,l)}y^{i_{\tau(1)}}\otimes\cdots\otimes y^{i_{\tau(k+l)}},
  $$
  where $\text{Sh}(k,l)$ denotes the set of all $(k,l)$-shuffles.
\end{rmk}

There is a canonical (classical) fiberwise multiplication, denoted as $\cdot$, which makes $\mathcal{W}_{M,\R}$ an (infinite rank) algebra bundle over $M$. We will also consider differential forms with values in $\mathcal{W}_{M,\R}$, i.e., $\A_M^\bullet(\mathcal{W}_{M,\R})$. 
\begin{rmk}
In this subsection, we are only concerned with the classical geometry of Weyl bundles; the formal variable $\hbar$ is included in Definition \ref{definition: real-Weyl-bundle} for discussing their quantum geometry in later sections. 
\end{rmk}

\begin{defn}\label{defn:complexified-Weyl-bundle}
There are two important operators on $\A_M^\bullet(\mathcal{W}_{M,\R})$ which are $\A_M^\bullet$-linear:\footnote{The Einstein summation rule will be used throughout this paper.}
\begin{equation*}\label{equation: delta-delta-star}
\delta a=dx^k\wedge\frac{\partial a}{\partial y^k},\quad \delta^*a=y^k\cdot \iota_{\partial_{x^k}}a.
\end{equation*}
Here $\iota_{\partial_{x^k}}$ denotes the contraction of differential forms by the vector field $\frac{\partial}{\partial x^k}$. We can normalize the operator $\delta^*$ by letting $\delta^{-1} := \frac{1}{l+m}\delta^*$ when acting on the monomial
$
y^{i_1}\cdots y^{i_l}dx^{j_1}\wedge\cdots\wedge dx^{j_m}.
$
Then for any form $a\in\A_M^\bullet(\mathcal{W}_{M,\R})$, we have
\begin{equation*}\label{equation: Hodge-de-Rham-decomposition-Moyal}
 a=\delta\delta^{-1}a+\delta^{-1}\delta a+a_{00},
\end{equation*}
where $a_{00}$ denotes the constant term, i.e., the term without any $dx^i$'s or $y^j$'s in $a$. 
\end{defn}

\begin{defn}
For a K\"ahler manifold $X$, we define the {\em holomorphic and anti-holomorphic Weyl bundles} respectively by
\begin{align*}
\mathcal{W}_X := \widehat{\Sym}(T^*X)[[\hbar]],\quad \overline{\mathcal{W}}_X := \widehat{\Sym}(\overline{T^*X})[[\hbar]],
\end{align*}
where $T^*X$ and $\overline{T^*X}$ are the holomorphic and anti-holomorphic cotangent bundles of $X$ respectively. With respect to a local holomorphic coordinate system $\{z^1,\dots, z^n\}$, we let $y^i$'s and $\bar{y}^j$'s denote the local frames of $T^*X$ and $\overline{T^*X}$ respectively. A local section of the complexification $\W_{X,\mathbb{C}} := \W_{X,\R} \otimes_\mathbb{R} \mathbb{C}$ of the real Weyl bundle is then of the form:
$$
 \sum_{k,m\geq 0}\sum_{i_1,\dots,i_m \geq0}\sum_{j_1,\dots,j_l \geq0}a_{k,i_1,\cdots,i_m,j_1,\cdots,j_l} \hbar^k y^{i_1}\cdots y^{i_m}\bar{y}^{j_1}\cdots\bar{y}^{j_l},
$$
from which we see that $\W_{X,\mathbb{C}}= \mathcal{W}_X\otimes_{\mathcal{C}^{\infty}_X}\overline{\mathcal{W}}_{X}$. The {\em symbol map}
\begin{equation*}\label{equation: definition-symbol-map}
\sigma: \A_X^\bullet\otimes\W_{X,\mathbb{C}}\rightarrow \A_X^\bullet[[\hbar]]
\end{equation*}
is defined by setting $y^i$'s and $\bar{y}^j$'s to be $0$.
\end{defn}

\begin{notn}
We will use the notation $\mathcal{W}_{p,q}$ to denote the component of $\mathcal{W}_{X,\mathbb{C}}$ of type $(p,q)$. 
Also we will often abuse the names ``Weyl bundle'' and ``symbol map'' when there is no ambiguity. 
\end{notn}

We introduce several operators on $\A_X^\bullet(\mathcal{W}_{X,\mathbb{C}})$, similar to those in Definition \ref{defn:complexified-Weyl-bundle}. 
\begin{defn}
There are 4 natural operators acting as derivations on $\A_X^\bullet(\mathcal{W}_{X,\mathbb{C}})$:
\begin{align*}
\delta^{1,0} a  = dz^i\wedge\frac{\partial a}{\partial y^i},\quad 
\delta^{0,1}a  = d\bar{z}^j\wedge\frac{\partial a}{\partial\bar{y}^j},
\end{align*}
as well as
\begin{align*}
(\delta^{1,0})^*a  = y^k\cdot \iota_{\partial_{z^k}}a, \quad
(\delta^{0,1})^*a  = \bar{y}^j\cdot \iota_{\partial_{\bar{z}^j}}a.
\end{align*}
We define the operators $(\delta^{1,0})^{-1}$ and $(\delta^{0,1})^{-1}$ by normalizing $(\delta^{1,0})^{*}$ and $(\delta^{1,0})^{*}$ respectively: 
\begin{equation*}\label{equation: delta-1-0-inverse}
 (\delta^{1,0})^{-1}:=\frac{1}{p_1+p_2}(\delta^{1,0})^*\ \text{on $\A_X^{p_1,q_1}(\mathcal{W}_{p_2,q_2})$},
\end{equation*}
\begin{equation*}\label{equation: delta-0-1-inverse}
 (\delta^{0,1})^{-1}:=\frac{1}{q_1+q_2}(\delta^{0,1})^*\ \text{on $\A_X^{p_1,q_1}(\mathcal{W}_{p_2,q_2})$}.
\end{equation*}
\end{defn}

\begin{rmk}
	We use the same notation for the operator $(\delta^{1,0})^{-1}$ as in Fedosov's original paper \cite{Fed}, although it could be confusing since it is actually {\em not} inverse to $\delta^{1,0}$.  
\end{rmk}

\begin{lem}
We have the following identities:
\begin{align*}
\delta & = \delta^{1,0} + \delta^{0,1},\\
\text{id}-\pi_{0,*} & = \delta^{1,0}\circ(\delta^{1,0})^{-1}+(\delta^{1,0})^{-1}\circ\delta^{1,0},\\
\text{id}-\pi_{*,0} & = \delta^{0,1}\circ(\delta^{0,1})^{-1}+(\delta^{0,1})^{-1}\circ\delta^{0,1},
\end{align*}
where $\pi_{0,*}$ and $\pi_{*,0}$ denote the natural projections from $\A_X^\bullet(\mathcal{W}_{X,\mathbb{C}})$ to $\A_X^{0,\bullet}(\overline{\mathcal{W}}_X)$ and  $\A_X^{\bullet,0}(\mathcal{W}_X)$ respectively.
\end{lem}

\subsection{A reformulation of Kapranov's $L_{\infty}$ structure on a K\"ahler manifold}\label{section: L-infty-structure}
\

\noindent In this subsection, we reformulate Kapranov's $L_\infty$-algebra structure \cite{Kapranov} on a K\"ahler manifold $X$ in terms of the holomorphic Weyl bundle $\W_X$ on $X$. Let us start with Kapranov's original theorem:
\begin{thm}[Theorem 2.6 and Reformulation 2.8.1 in \cite{Kapranov}]\label{thm:Kapranov-L-infinity}
	Let $X$ be a K\"ahler manifold. Then there exist
	$$
	R_n^*\in\mathcal{A}^{0,1}_X(\Hom(T^*X,\Sym^n(T^*X))),\qquad n\geq 2
	$$
	such that their extensions $\tilde{R}^*_n$ to the holomorphic Weyl bundle $\W_X$ by derivation satisfy
	$$\left(\bar{\partial}+\sum_{n\geq 2}\tilde{R}_n^*\right)^2 = 0,$$
	or equivalently,
\begin{equation}\label{equation: square-zero-0-1-part}
	\bar{\partial}\tilde{R}_n^*+\sum_{j+k=n+1}\tilde{R}_j^*\circ\tilde{R}_k^* = 0
\end{equation}
	for any $n \geq 2$.
\end{thm}

The $R_n^*$'s are defined as partial transposes of the higher covariant derivatives of the curvature tensor:
$$
R_2^*=\frac{1}{2}R_{i\bar{j}k}^m d\bar{z}^j\otimes (y^iy^k\otimes\partial_{z^m}),\qquad R_n^*=(\delta^{1,0})^{-1}\circ\nabla^{1,0}(R_{n-1}^*),
$$
where, by abuse of notations, we use $\nabla$ to denote the Levi-Civita connection on the (anti)holomorphic (co)tangent bundle of $X$, as well as their tensor products including the Weyl bundles. 
We can write these $R_n^*$'s locally in a more consistent way as:
\begin{equation}\label{equation: terms-L-infty-structure}
 R_n^*=R_{i_1\cdots i_n,\bar{l}}^j d\bar{z}^l\otimes (y^{i_1}\cdots y^{i_n}\otimes \partial_{z^j}).
\end{equation}
Readers are referred to \cite[Section 2.5]{Kapranov} for a detailed exposition. 

The $L_\infty$ relations \eqref{equation: square-zero-0-1-part} can be reformulated as the flatness of a natural connection $D_K$ (where the subscript ``K'' stands for ``Kapranov'') on the holomorphic Weyl bundle, whose $(0,1)$-part is exactly the differential operator in Theorem \ref{thm:Kapranov-L-infinity}:
\begin{prop}
	The operator
	\begin{equation*}\label{eqn:classical-flat-connection-holomorphic}
	D_K=\nabla-\delta^{1,0}+\sum_{n\geq 2}\tilde{R}_n^*
	\end{equation*}
	defines a flat connection on the holomorphic Weyl bundle $\mathcal{W}_X$, which is compatible with the classical (commutative) product.
\end{prop}
\begin{proof}
	To show the vanishing of the $(2,0), (0,2)$ and $(1,1)$ parts of $D_K^2$, we let $D_K^{1,0}$ and $D_K^{0,1}$ denote the $(1,0)$ and $(0,1)$ parts of the connection $D_K$ respectively, and note that
	the Levi-Civita connection $\nabla$ on $\mathcal{W}_X$ has the type decomposition $\nabla=\nabla^{1,0}+\bar{\partial}$. 
	First, it is clear that $D_K^{0,1}$ is exactly Kapranov's differential in Theorem \ref{thm:Kapranov-L-infinity}. Thus the vanishing of the $(0,2)$ part of $D_K^2$ follows from Theorem \ref{thm:Kapranov-L-infinity}. 
	Next, the vanishing of the square of $D_K^{1,0}=\nabla^{1,0}-\delta^{1,0}$ follows from the following computation:
	\begin{align*}
	(\nabla^{1,0}-\delta^{1,0})^2
	& = (\nabla^{1,0})^2+(\delta^{1,0})^2-(\nabla^{1,0}\circ\delta^{1,0}+\delta^{1,0}\circ\nabla^{1,0})\\
	& = -(\nabla^{1,0}\circ\delta^{1,0}+\delta^{1,0}\circ\nabla^{1,0}) = 0,
	\end{align*}
	where the last equality follows from the torsion-freeness of $\nabla$. Finally, for the $(1,1)$ part, we have
	$$
	D_K^{1,0}\circ D_K^{0,1}+D_K^{0,1}\circ D_K^{1,0}=(\nabla^{1,0}-\delta^{1,0})\left(\sum_{n\geq 2}R_n^*\right)+\nabla^2.
	$$
	Also, $\delta^{1,0}(R_2^*)=\nabla^2$. So we only need to show that $\nabla^{1,0}(R_n^*)=\delta^{1,0}(R_{n+1}^*)$ for $n\geq 2$.
	Recall that $R_n^*$ is inductively defined by $R_{n+1}^*=(\delta^{1,0})^{-1}\circ\nabla^{1,0}(R_n^*)$ for $n\geq 2$. It follows that 
	\begin{align*} 
 	\delta^{1,0}(R_{n+1}^*) & = \delta^{1,0}\circ(\delta^{1,0})^{-1}\left(\nabla^{1,0}(R_n^*)\right)\\
 	& = \nabla^{1,0}(R_n^*)-(\delta^{1,0})^{-1}\circ\delta^{1,0}\left(\nabla^{1,0}(R_n^*)\right)\\
 	& = \nabla^{1,0}(R_n^*),
	\end{align*}
	as desired; here the last equality follows from the fact that $\nabla^{1,0}(R_n^*)$ is symmetric in all lower subscripts, which was shown in \cite{Kapranov}.
\end{proof}


Now if $\alpha$ is a local flat section of $\mathcal{W}_X$ under $D_K$, it is easy to see that $\sigma(\alpha)$ must be a holomorphic function. Actually, the symbol map defines an isomorphism: 
\begin{prop}\label{proposition:flat-sections-holomorphic-Weyl}
	The space of (local) flat sections of the holomorphic Weyl bundle with respect to the connection $D_K$ is isomorphic to the space of holomorphic functions. More precisely, the symbol map
	$
	\sigma: \Gamma^{\text{flat}}(U,\mathcal{W}_X)\rightarrow\mathcal{O}_X(U)[[\hbar]]
	$
	is an isomorphism for any open subset $U \subset X$.
\end{prop}

To prove this proposition, we mimic Fedosov's arguments in \cite{Fed}. First we need a lemma:
\begin{lem}\label{lemma: symmetry-property-Taylor-expansion}
	Let $\alpha_0\in\A_X^{0,q}$ be a smooth differential form on $X$ of type $(0,q)$. Define a sequence $\{\alpha_k\}$ by $\alpha_k := ((\delta^{1,0})^{-1}\circ\nabla^{1,0})^k(\alpha_0) \in \A_X^{0,q}(\mathcal{W}_{k,0})$ for $k \geq 1$. Then we have $(\delta^{1,0}\circ\nabla^{1,0})(\alpha_k)=0$ for all $k\geq 1$, and hence $D_K^{1,0}\left(\sum_{k\geq 0}\alpha_k\right)=0$. 
\end{lem}
\begin{proof}
Without loss of generality we can assume that $q=0$, i.e., $\alpha_0$ is a function on $X$. Let us write $\alpha_k$ as
$
\alpha_k=a_{i_1\cdots i_k}y^{i_1}\otimes\cdots\otimes y^{i_k},
$
where the coefficients $a_{i_1\cdots i_k}$ are totally symmetric with respect to all indices. We also write 
$
\nabla^{1,0}(\alpha_k)=b_{i_1\cdots i_{k+1}}dz^{i_1}\otimes (y^{i_2}\otimes\cdots\otimes y^{i_{k+1}}).
$
We will show by induction that the coefficients $b_{i_1\cdots i_{k+1}}$ are totally symmetric with respect to all indices $i_1,\cdots,i_{k+1}$. This is clearly true for $k=1$. For general $k\geq 1$, it is clear that $\nabla^{1,0}(\alpha_k)=b_{i_1\cdots i_{k+1}}dz^{i_1}\otimes (y^{i_2}\otimes\cdots\otimes y^{i_{k+1}})$ is symmetric in $i_2,\cdots,i_{k+1}$, so we only need to show that it is symmetric in the first two indices $i_1$ and $i_2$. We know from construction that  $\nabla^{1,0}(\alpha_{k-1}) = 
a_{i_1\cdots i_k}dz^{i_1}\otimes y^{i_2}\cdots\otimes y^{i_k}$.
Since $(\nabla^{1,0})^2 = 0$, we have
\begin{equation}\label{eqn:nabla-square-vanish}
-dz^{i_1}\wedge \nabla^{1,0}(a_{i_1\cdots i_k}y^{i_2}\otimes\cdots\otimes y^{i_k})
= \nabla^{1,0}(a_{i_1\cdots i_k}dz^{i_1}\otimes y^{i_2}\cdots\otimes y^{i_k})
= 0.
\end{equation}
We then compute the covariant derivative:
\begin{align*}
 \nabla^{1,0}(\alpha_k)=&\nabla^{1,0}(a_{i_1\cdots i_k}y^{i_1}\otimes\cdots\otimes y^{i_k})\\
 =&y^{i_1}\otimes\nabla^{1,0}(a_{i_1\cdots i_k}y^{i_2}\otimes\cdots\otimes y^{i_k})+\nabla^{1,0}(y^{i_1})\otimes(a_{i_1\cdots i_k}y^{i_2}\otimes\cdots\otimes y^{i_k}).
\end{align*}
Since the Levi-Civita connection is torsion-free, the second term in the last expression is symmetric in the first two indices. Now the first term has the same symmetric property by \eqref{eqn:nabla-square-vanish}. Hence we have $(\delta^{1,0}\circ\nabla^{1,0})(\alpha_k)=0$.

Now we have
\begin{align*}
\nabla^{1,0}(\alpha_k)
& = (\delta^{1,0}\circ(\delta^{1,0})^{-1}+(\delta^{1,0})^{-1}\circ\delta^{1,0})(\nabla^{1,0}(\alpha_k))\\
& = (\delta^{1,0}\circ(\delta^{1,0})^{-1})(\nabla^{1,0}(\alpha_k))
= \delta^{1,0}(\alpha_{k+1}).
\end{align*}
Therefore
\begin{align*}
D_K^{1,0}\left(\sum_{k\geq 0}\alpha_k\right)
& = (\nabla^{1,0}-\delta^{1,0})\left(\alpha_0+\alpha_1+\alpha_2+\cdots\right)\\
& = -\delta^{1,0}(\alpha_0)+(\nabla^{1,0}(\alpha_0)-\delta^{1,0}(\alpha_1))+(\nabla^{1,0}(\alpha_1)-\delta^{1,0}(\alpha_2))+\cdots\\
& = 0.
\end{align*}
\end{proof}




\begin{proof}[Proof of Proposition \ref{proposition:flat-sections-holomorphic-Weyl}]
The injectivity follows by observing that any nonzero section $s$ of $\mathcal{W}_X$ with zero constant term cannot be flat since $\delta^{1,0}(s_0)$, where $s_0$ is the leading term (i.e. of least weight) in $s$, is of smaller weight and nonzero.

For the surjectivity, we claim that given a holomorphic function $f$, the section
\begin{equation}\label{equation: flat-section-D_K}
J_f := \sum_{k\geq 0}((\delta^{1,0})^{-1}\circ\nabla^{1,0})^k(f),
\end{equation}
with leading term $f$ is flat with respect to $D_K$. From Lemma \ref{lemma: symmetry-property-Taylor-expansion}, we know that $D_K^{1,0}(J_f)=0$. It follows that $D_K^{1,0}\circ D_K^{0,1}(J_f)=-D_K^{0,1}\circ D_K^{1,0}(J_f)=0$.
Observe that $\sigma(D_K^{0,1}(J_f))=\bar{\partial}f=0$. Thus, by the same reason as in the proof of the injectivity of $\sigma$, we must have $D_K^{0,1}(J_f)=0$. 
\end{proof}

Similar arguments give the following proposition, which will be useful later:
\begin{prop}\label{proposition: flat-section-closed-0-q-form}
Let $\alpha$ be a $\bar{\partial}$-closed $(0,q)$ form on $X$. Then the section
$$
\alpha+\sum_{k\geq 1}((\delta^{1,0})^{-1}\circ\nabla^{1,0})^k(\alpha)
$$
is flat with respect to the connection $D_K$.
\end{prop}


\subsection{Fedosov quantization from Kapranov's $L_\infty$ structure}\label{section: Fedosov-quantization}
\

\noindent In \cite{Fed}, Fedosov gave an elegant geometric construction of deformation quantization on a general symplectic manifold (see also \cite{Fedbook, Fed-index}). 
The goal of this section is to show that Fedosov's quantization arises naturally as a quantization of Kapranov's $L_\infty$ structure. The key is to adapt Fedosov's construction in the K\"ahler setting by incorporating the complex structure, or complex polarization, on $X$. We first recall the notion of deformation quantization:
\begin{defn}
A deformation quantization (star product) of smooth functions on a symplectic manifold $X$ is an associative multiplication $\star$ on $C^\infty(X)[[\hbar]]$ of the following form:
$$
f\star g=f\cdot g+\sum_{k\geq 1}\hbar^k C_k(f,g),\  \text{for}\  f,g\in C^\infty(X),
$$
where the $C_k$'s are bi-differential operators on $C^\infty(X)$ such that
$$
\{f,g\}=\lim_{\hbar\rightarrow 0}\frac{1}{\hbar}\left(f\star g-g\star f\right).
$$
A star product $\star$ on a K\"ahler manifold is said to be of {\em Wick type} if the bi-differential operators $C_k(-,-)$ take holomorphic derivatives to $f$ and anti-holomorphic derivatives to $g$.  
\end{defn}

We begin by considering the complexified Weyl bundle $\mathcal{W}_{X,\mathbb{C}}$, equipped with the {\em fiberwise Wick product} induced by the K\"ahler form $\omega$: if $\alpha,\beta$ are sections of $\mathcal{W}_{X,\mathbb{C}}$, the Wick product is explicitly defined by:
\begin{equation*}\label{equation: Wick-product}
 \alpha\star\beta:=\sum_{k\geq 0}\frac{\hbar^k}{k!}\cdot\omega^{i_1\bar{j}_1}\cdots\omega^{i_k\bar{j}_k}\cdot\frac{\partial^k\alpha}{\partial y^{i_1}\cdots\partial y^{i_k}}\frac{\partial^k\beta}{\partial \bar{y}^{j_1}\cdots\partial \bar{y}^{j_k}}.
\end{equation*}
With respect to this product, the operator $\delta$ on $\mathcal{W}_{X,\C}$ can be expressed as:
$$
\delta =\frac{1}{\hbar}\left[\omega_{i\bar{j}}(dz^i\otimes\bar{y}^j-d\bar{z}^j\otimes y^i),-\right]_{\star},
$$
where $[-,-]_{\star}$ denotes the bracket associated to the Wick product.
\begin{defn}
A connection on $\W_{X,\C}$ of the form 
$$
D=\nabla-\delta+\frac{1}{\hbar}[I,-]_{\star}
$$
is called a {\em Fedosov abelian connection} if $D^2=0$.  Here $\nabla$ is the Levi-Civita connection, and $I\in\A^1(X,\W_{X,\C})$ is a $1$-form valued section of $\W_{X,\C}$. 
\end{defn}

Recall the following calculation in \cite{Bordemann}:
\begin{lem}[Proposition 4.1 in \cite{Bordemann}]\label{lemma: curvature-on-Weyl-bundle}
 The curvature of the Levi-Civita connection on the Weyl bundle is given by
 \begin{equation*}\label{equation: curvature-on-Weyl-bundle}
  \nabla^2=\frac{1}{\hbar}[R_\nabla,-]_{\star},
 \end{equation*}
 where $R_\nabla:=\sqrt{-1}\cdot R_{i\bar{j}k\bar{l}}dz^i\wedge d\bar{z}^j\otimes y^k\bar{y}^l$.
\end{lem}

A simple computation together with Lemma \ref{lemma: curvature-on-Weyl-bundle} shows that the flatness of $D$ is equivalent to the {\em Fedosov equation}:
\begin{equation}\label{eqn:Fedosov-equation-Wick}
\nabla I-\delta I+\frac{1}{\hbar}I\star I+ R_\nabla=\alpha,
\end{equation}
where $\alpha=\sum_{k\geq 1}\hbar^k \omega_k\in\A_X^{2}[[\hbar]]$ is a closed formal $2$-form on $X$.
Letting $\gamma:=I+\omega_{i\bar{j}}(dz^i\otimes \bar{y}^j-d\bar{z}^j\otimes y^i)$ and $\omega_\hbar:=-\omega+\alpha$, then equation \eqref{eqn:Fedosov-equation-Wick} is equivalent to
\begin{equation}\label{equation: Fedosov-equation-gamma}
 \nabla \gamma+\frac{1}{\hbar}\gamma\star\gamma+ R_\nabla=\omega_\hbar.
\end{equation}





We will show that the flat connection $D_K$ on $\mathcal{W}_X$ gives rise to a Fedosov abelian connection $D_F$ (where the subscript ``F'' stands for ``Fedosov''), which is a {\em quantization} (or {\em quantum extension}) because the Wick product $\star$ is non-commutative. First recall that $$D_K=\nabla-\delta^{1,0}+\sum_{n\geq 2}\tilde{R}_n^*,$$
where each $\tilde{R}_n^*$ is an $\A_X^\bullet$-linear derivation on $\mathcal{W}_{X,\C}$ via $R_n^*\in\A_X^{0,1}\left(\Sym^n(T^*X)\otimes TX\right)$ (see equation \eqref{equation: terms-L-infty-structure}). 
Consider the $\A_X^{\bullet}$-linear operator
$$L:\A_X^\bullet\left(\widehat{\Sym}(T^*X)\otimes TX\right)\rightarrow \A_X^\bullet\left(\widehat{\Sym}(T^*X)\otimes \overline{T^*X}\right)$$
given by ``lifting the last subscript'' using the K\"ahler form $\omega$. Then we can define
$$
I_n:=L(R_n^*)=R_{i_1\cdots i_n,\bar{l}}^j\omega_{j\bar{k}}d\bar{z}^l\otimes (y^{i_1}\cdots y^{i_n}\bar{y}^{k})\in\mathcal{A}_X^{0,1}(\mathcal{W}_{X,\mathbb{C}}). 
$$
\begin{lem}\label{lemma: commutation-relations}
We have
\begin{equation}\label{equation: connection-compatible-with-lifting-subscript}
 \tilde{R}_n^*=\frac{1}{\hbar}[I_n,-]_\star|_{\mathcal{W}_{X}};
\end{equation}
\begin{equation}\label{equation: connection-commutes-with-L}
 \nabla\circ L=L\circ\nabla;
\end{equation}
\begin{equation}\label{equation: L-commutes-with-delta}
 L\circ\delta^{1,0}=\delta^{1,0}\circ L;
\end{equation}
\begin{equation}\label{equation: curvature-first-term-L-infty}
 \delta^{1,0}\circ L(R_2^*)=R_\nabla;
\end{equation}
\begin{equation}\label{equation: L-Lie-algebra-homomorphism}
 L([R_m^*,R_n^*])=[L(R_m^*),L(R_n^*)]_{\star}.
\end{equation}
\end{lem}
\begin{proof}
The construction of $I_n$'s implies equation \eqref{equation: connection-compatible-with-lifting-subscript}.
Equation \eqref{equation: connection-commutes-with-L} follows from the fact that  $\nabla(\omega)=0$. Equation \eqref{equation: L-commutes-with-delta} is obvious.  Equation \eqref{equation: curvature-first-term-L-infty} follows from a straightforward computation:
\begin{align*}
 \delta^{1,0}\circ L(R_2^*)=&L\circ \delta^{1,0}\left(\frac{1}{2}R_{i\bar{j}k}^m d\bar{z}^j\otimes(y^iy^k\otimes\partial_{z^m})\right)\\
 =&L\left(R_{i\bar{j}k}^m dz^i\wedge d\bar{z}^j\otimes(y^k\otimes\partial_{z^m})\right)\\
 =&R_{i\bar{j}k}^m\omega_{m\bar{l}}dz^i\wedge d\bar{z}^j\otimes y^ky^l\\
 =&R_{i\bar{j}k}^m\sqrt{-1}g_{m\bar{l}}dz^i\wedge d\bar{z}^j\otimes y^ky^l\\
 =&R_\nabla.
\end{align*}
To show equation \eqref{equation: L-Lie-algebra-homomorphism}, notice that $L([R_m^*,R_n^*])=L\left(\tilde{R}_m^*(R_n^*)+\tilde{R}_n^*(R_m^*)\right)$.
We then have the explicit computation:
\begin{align*}
L\left(\tilde{R}_m^*(R_n^*)\right)
& = L\left(R_{i_1\cdots i_m,\bar{l}}^jd\bar{z}^l\otimes(y^{i_1}\cdots y^{i_m}\otimes\partial_{z^j})(R_{i'_1\cdots i'_n,\bar{l}'}^{j'}d\bar{z}^{l'}\otimes(y^{i'_1}\cdots y^{i'_n}\otimes\partial_{z^{j'}}))\right)\\
& = L\left(\sum_{1\leq \alpha\leq n}R_{i_1\cdots i_m,\bar{l}}^jR_{i'_1\cdots i'_n,\bar{l}'}^{j'}d\bar{z}^l\wedge d\bar{z}^{l'}\otimes(y^{i_1}\cdots y^{i_m}y^{i'_1}\cdots\widehat{y^{i'_\alpha}}\cdots y^{i'_n}\otimes\partial_{z^{j'}})\right)\\
& = \sum_{1\leq \alpha\leq n}\omega_{j'\bar{k}}R_{i_1\cdots i_m,\bar{l}}^{i'_\alpha}R_{i'_1\cdots i'_n,\bar{l}'}^{j'}d\bar{z}^l\wedge d\bar{z}^{l'}\otimes(y^{i_1}\cdots y^{i_m}y^{i'_1}\cdots\widehat{y^{i'_\alpha}}\cdots y^{i'_n}\bar{y}^k).
\end{align*}
On the other hand, we have
\begin{align*}
 &L(R_n^*)\star L(R_m^*)\\
 =&\left(R_{i'_1\cdots i'_n,\bar{l}'}^{j'}\omega_{j'\bar{k}'}d\bar{z}^{l'}\otimes(y^{i'_1}\cdots y^{i'_n} y^{\bar{k'}})\right)\star\left(R_{i_1\cdots i_m,\bar{l}}^j\omega_{j\bar{k}}d\bar{z}^l\otimes(y^{i_1}\cdots y^{i_m} y^{\bar{k}})\right)\\
 =&L(R_n^*)\cdot L(R_m^*)\\
 &\qquad + \sum_{1\leq \alpha\leq n}\omega_{j\bar{k}}\omega_{j'\bar{k}'}\omega^{i'_\alpha\bar{k}}R_{i_1\cdots i_m,\bar{l}}^jR_{i'_1\cdots i'_n,\bar{l}'}^{j'}d\bar{z}^{l'}\wedge d\bar{z}^{l}\otimes(y^{i_1}\cdots y^{i_m}y^{i'_1}\cdots\widehat{y^{i'_\alpha}}\cdots y^{i'_n}\bar{y}^{k'}) \\
 =&-L(R_m^*)\cdot L(R_n^*)\\
 &\qquad + \sum_{1\leq \alpha\leq n}(-1)\omega_{j'\bar{k}'}R_{i_1\cdots i_m,\bar{l}}^{i'_\alpha}R_{i'_1\cdots i'_n,\bar{l}'}^{j'}d\bar{z}^{l'}\wedge d\bar{z}^{l}\otimes(y^{i_1}\cdots y^{i_m}y^{i'_1}\cdots\widehat{y^{i'_\alpha}}\cdots y^{i'_n}\otimes\partial_{z^{j'}})\\
 =&-L(R_m^*)\cdot L(R_n^*)+L\left(\tilde{R}_m^*(R_n^*)\right).
\end{align*}
\end{proof}
Let $I := \sum_{n\geq 2}I_n$. Then
\begin{equation*}\label{equation: Fedosov-connection-Wick-type}
D_F:=\nabla-\delta+\frac{1}{\hbar}[I,-]_{\star}
\end{equation*}
defines a connection on $\mathcal{W}_{X,\mathbb{C}}$.
Equation \eqref{equation: connection-compatible-with-lifting-subscript} says that $D_F$ is an extension of the $D_K$, namely,
\begin{equation*}\label{equation: Fedosov-connection-equals-Kapranov-connection-holomorphic-Weyl}
D_F|_{\mathcal{W}_X}=D_K.
\end{equation*}

\begin{lem}\label{lemma: delta-0-1-annihilates-I}
We have $\delta^{0,1}I=0$.
\end{lem}
\begin{proof}
Recall that $I_2=\frac{\sqrt{-1}}{2}R_{i\bar{j}k\bar{l}}d\bar{z}^l\otimes y^iy^k\bar{y}^j$, where $R_{i\bar{j}k\bar{l}}$ is symmetric in $\bar{j}$ and $\bar{l}$, from which we know that $\delta^{0,1}I_2=0$. The statement for the general $I_n$'s follows from the iterative equation $I_n=(\delta^{1,0})^{-1}\circ\nabla^{1,0}(I_{n-1})$ and the commutativity relations
$[\delta^{0,1},(\delta^{1,0})^{-1}]=[\delta^{0,1},\nabla^{1,0}]=0$.
\end{proof}

Here comes the first main result of this paper:
\begin{thm}\label{theorem: quantization-of-Kapranov-is-Fedosov-connection}
The connection $D_F$ is flat. 
\end{thm}
\begin{proof}
Notice that $I=\sum_{n\geq 2}I_{n}$ is a $(0,1)$-form valued in $\mathcal{W}_{X,\mathbb{C}}$. We only need to show the vanishing of the $(2,0)$, $(0,2)$ and $(1,1)$ parts of $D_F^2$. 
The vanishing of the $(2,0)$ part of $D_F^2$ follows from that of $D_K$.  The $(0,2)$ part of $D_F^2$ is given by 
\begin{equation*}
\frac{1}{\hbar}\left[\nabla^{0,1}I-\delta^{0,1}I+\frac{1}{\hbar}[I,I]_{\star},-\right]_{\star}.
\end{equation*}
By Lemma \ref{lemma: commutation-relations} and the flatness of $D_K$, we have
\begin{align*}
\nabla^{0,1}I-\delta^{0,1}I+\frac{1}{\hbar}[I,I]_{\star}
=  \nabla^{0,1}I+\frac{1}{\hbar}[I,I]_{\star}
=  L\left(\nabla^{0,1}(\sum_{k\geq 2}\tilde{R}_k^*)+\left[\sum_{k\geq 2}\tilde{R}_k^*, \sum_{k\geq 2}\tilde{R}_k^*\right]\right)
=  0;
\end{align*}
here the first equality follows from Lemma \ref{lemma: delta-0-1-annihilates-I}.   
The $(1,1)$ part of $D_F^2$ is given by 
\begin{align*}
&\hspace{5mm}\nabla^2+\frac{1}{\hbar}\left[\nabla^{1,0}I-\delta^{1,0}I, -\right]_{\star}\\
& = \frac{1}{\hbar}\left[R_\nabla+\nabla^{1,0}I-\delta^{1,0}I,\ -\right]_{\star}\\
& = \frac{1}{\hbar}\left[R_\nabla+(\nabla^{1,0}-\delta^{1,0})\circ L\left(\sum_{k\geq 2}R_k^*\right),\ -\right]_\star\\
& \overset{(*)}{=} \frac{1}{\hbar}\left[\delta^{1,0}\circ L(R_2^*)+(\nabla^{1,0}-\delta^{1,0})\circ L\left(\sum_{k\geq 2}R_k^*\right),\ -\right]_\star\\
& = \frac{1}{\hbar}\left[L\circ\delta^{1,0}(R_2^*)+L\circ(\nabla^{1,0}-\delta^{1,0})\left(\sum_{k\geq 2}R_k^*\right),\ -\right]_\star\\
& = \frac{1}{\hbar}\left[L\left(\delta^{1,0}(R_2^*)-\delta^{1,0}(R_2^*)+(\nabla^{1,0}R^*_2-\delta^{1,0}R^*_3)+\cdots+(\nabla^{1,0}R^*_k-\delta^{1,0}R^*_{k+1})+\cdots\right),-\right]_\star\\
& = 0;
\end{align*}
where we have used equation \eqref{equation: curvature-first-term-L-infty} in the equality $(*)$, and also the fact that $L$ commutes with both $\delta^{1,0}$ and $\nabla^{1,0}$. Hence $D_F$ is a Fedosov abelian connection. A simple computation shows that $I$ actually satisfies the Fedosov equation \eqref{eqn:Fedosov-equation-Wick} with $\alpha = 0$:
$$
\nabla I-\delta I+\frac{1}{\hbar}I\star I+ R_\nabla=0. 
$$
\end{proof}

\begin{rmk}
 It is worth pointing out that this flat connection $D_F$ only has a ``classical'' part, i.e., the sections $I_n\in\A_X^{\bullet}(\mathcal{W}_{X,\C})$ have no higher order terms in the $\hbar$-power expansion. This is very different from Fedosov's original solutions of his equation. 
\end{rmk}

More generally, given any closed formal $2$-form $\alpha$ on $X$ with $[\alpha]\in\hbar H^2_{dR}(X)[[\hbar]]$ of type $(1,1)$, the following theorem produces an explicit solution of the Fedosov equation: \eqref{eqn:Fedosov-equation-Wick}:
\begin{thm}\label{proposition: Fedosov-connection-general}
 Let $\alpha$ be a representative of a formal cohomology class in $\hbar H^2_{dR}(X)[[\hbar]]$ of type $(1,1)$. Then there exists a solution of the Fedosov equation of the form $I_\alpha = I+ J_\alpha \in \mathcal{A}_X^{0,1}(\mathcal{W}_{X,\mathbb{C}})$:
 \begin{equation*}
\nabla I_\alpha - \delta I_\alpha + \frac{1}{\hbar} I_\alpha\star I_\alpha + R_\nabla=\alpha.
\end{equation*}
We denote the corresponding Fedosov abelian connection by $D_{F,\alpha}$. 
\end{thm}
\begin{proof}
  The $\partial\bar{\partial}$-lemma guarantees the local existence of a formal function $g$ such that $\alpha=\bar{\partial}{\partial}(g)$. We define a section $J_\alpha$ of $\mathcal{A}_X^{0,1}(\mathcal{W}_X)$ by
 \begin{equation*}\label{equation: I-alpha-formula}
 J_\alpha:=\sum_{k\geq 1}\left((\delta^{1,0})^{-1}\circ\nabla^{1,0}\right)^k(\bar{\partial}g).
 \end{equation*} 
 Such a function $g$ is unique up to a sum of purely holomorphic and purely anti-holomorphic functions. It follows that $J_\alpha$ is independent of the choice of $g$. In particular, this implies that these local sections $J_\alpha$'s patch together to give a global section over $X$.
 
 By Proposition \ref{proposition: flat-section-closed-0-q-form}, $\bar{\partial}g + J_\alpha$ is closed under $D_K$, so it is also closed under $D_F$ by the fact that $D_F|_{\mathcal{W}_X}=D_K$. Now
 \begin{align*}
 D_F(\bar{\partial}g + J_\alpha)
 = \partial\bar{\partial}g+\nabla J_\alpha-\delta(J_\alpha)+\frac{1}{\hbar}[I, J_\alpha]_{\star}
= -\alpha+\nabla J_\alpha-\delta J_\alpha +\frac{1}{\hbar}[I, J_\alpha]_{\star},
 \end{align*}
 so $\nabla J_\alpha-\delta J_\alpha + \frac{1}{\hbar}[I, J_\alpha]_{\star}=\alpha$.
 Together with Theorem \ref{theorem: quantization-of-Kapranov-is-Fedosov-connection} and the fact that $J_\alpha$ has only holomorphic components in $\mathcal{W}_{X,\mathbb{C}}$ which implies that $J_\alpha \star J_\alpha=0$ by type reasons, we deduce that
 $\nabla I_\alpha -\delta I_\alpha + \frac{1}{\hbar} I_\alpha \star I_\alpha + R_\nabla
 = \nabla J_\alpha - \delta J_\alpha + \frac{1}{\hbar}\left([I, J_\alpha]_{\star} + J_\alpha\star J_\alpha\right) = \alpha$.
\end{proof}

We now briefly recall the construction of star products from Fedosov connections. 
\begin{prop}[Theorem 3.3 in \cite{Fed}]\label{proposition: iteration-equation-quantum-flat-section}
	Let $D=\nabla-\delta+\frac{1}{\hbar}[I,-]_{\star}$ be a Fedosov abelian connection. Then there is a one-to-one correspondence induced by the symbol map $\sigma$:
	$$
	\sigma: \Gamma^{flat}(X,\mathcal{W}_{X,\C})\xrightarrow{\sim} C^\infty(X)[[\hbar]].
	$$
	The flat section $O_f$ corresponding to $f\in C^\infty(X)$ is the unique solution of the iterative equation: 
	\begin{equation}\label{equation: iteration-equation-quantum-flat-section}
	O_f=f+\delta^{-1}\left(\nabla O_f+\frac{1}{\hbar}[I,O_f]_{\star}\right).
	\end{equation}
\end{prop}

For the Fedosov connection $D_{F,\alpha}$ defined in Theorem \ref{proposition: Fedosov-connection-general}, the associated deformation quantization (star product) $\star_{\alpha}$ on $C^\infty(X)[[\hbar]]$ is given by 
\begin{equation}\label{equation: formal-star-product}
f\star_{\alpha} g:=\sigma(O_f\star O_g).
\end{equation}
\begin{defn}\label{definition: Wick-type-deformation-quantization}
	We say that a deformation quantization $(C^\infty(X)[[\hbar]],\star)$ is of {\em Wick type} (or {\em with separation of variables}) if we have $f \star g = f\cdot g$ whenever $f$ is antiholomorphic or $g$ is holomorphic.
	This is equivalent to requiring the corresponding bi-differential operators $C_i(-,-)$ to take holomorphic and anti-holomorphic derivatives of the first and second arguments respectively. 
\end{defn}


\begin{prop}\label{proposition: Fedosov-quantization-1-1-class-Wick-type}
	For every closed formal differential form $\alpha$ of type $(1,1)$, the star product $\star_{\alpha}$ defined in \eqref{equation: formal-star-product} is of Wick type.  
\end{prop}
\begin{proof}
	To show that $\star_{\alpha}$ is of Wick type, we only need to show that if both $f$ and $g$ are (anti-)holomorphic functions, then $f\star_{\alpha} g=f\cdot g$.
	
	If $f,g$ are holomorphic, then both $O_f$ and $O_g$ are sections of the holomorphic Weyl bundle $\W_X$. Now $D_{F,\alpha}=\nabla-\delta+\frac{1}{\hbar}[I_\alpha,-]_\star$ where $I_\alpha = I + J_\alpha$ and since $J_\alpha\in\A_X^{0,1}\otimes\mathcal{W}_X$, we have $D_{F,\alpha}|_{\mathcal{W}_X}=D_F|_{\mathcal{W}_X}=D_K$.
	It follows that $O_f=J_f$ (where $J_f$ is defined in \eqref{equation: flat-section-D_K}) and must be of the desired type. Thus we have $\sigma(O_f\star O_g)=f\cdot g$ by type reasons. 
	
	If $f, g$ are anti-holomorphic functions, then we are in the opposite situation. By a simple induction using equation \eqref{equation: iteration-equation-quantum-flat-section}, we can see that both $O_f$ and $O_g$ do {\em not} contain any term that has only holomorphic components in $\mathcal{W}_{X,\mathbb{C}}$. Hence it also follows that $\sigma(O_f\star O_g)=f\cdot g$. 
\end{proof}

\begin{rmk}
	The fact that $O_f=J_f$ for any (local) holomorphic function $f$ says that holomorphic functions do {\em not} receive any quantum corrections in these Fedosov quantizations. 
\end{rmk}


\subsubsection{Gauge fixing conditions and comparison with previous works}
\ 

\noindent Fedosov's original solutions \cite{Fed} of his equation satisfy the gauge fixing condition that $\delta^{-1}(I)=0$. On the other hand, by a simple type reason argument, it is easy to see that our solutions of the Fedosov equation satisfy instead the following gauge fixing condition:
$$
(\delta^{1,0})^{-1}(I)=0.
$$
This condition alone cannot guarantee uniqueness, but we have the following:
\begin{prop}\label{proposition-fedosov-gauge-conditions}
	The solution $I$ of the Fedosov equation \eqref{eqn:Fedosov-equation-Wick} satisfying the two conditions:
	\begin{equation}\label{equation: holomorphic-gauge-fixing}
	(\delta^{1,0})^{-1}(I)=0,
	\end{equation}
	\begin{equation}\label{equation: normalization-condition}
	\pi_{0,*}(I)=0
	\end{equation}
	is unique.  
\end{prop}
\begin{proof}
	Equation \eqref{equation: normalization-condition} implies that $(\delta^{1,0}\circ(\delta^{1,0})^{-1}+(\delta^{1,0})^{-1}\circ\delta^{1,0})(I)=I$.
	Together with the gauge fixing condition \eqref{equation: holomorphic-gauge-fixing}, we have
	$I=(\delta^{1,0})^{-1}\circ\delta^{1,0} (I)$.
	By applying the operator $(\delta^{1,0})^{-1}$ to equation \eqref{eqn:Fedosov-equation-Wick}, we see that $I$ must satisfy the following iterative equation:
	\begin{align*}
	I = (\delta^{1,0})^{-1}\left(\nabla I+\frac{1}{\hbar}[I,I]_{\star}+R_\nabla-\omega_{\hbar}-\delta^{0,1}I\right) 
	= (\delta^{1,0})^{-1}\left(\nabla I+\frac{1}{\hbar}[I,I]_{\star}+R_\nabla-\omega_{\hbar}\right),
	\end{align*}
	where we have used the fact that the two operators $\delta^{0,1}$ and $(\delta^{1,0})^{-1}$ commute with each other in the second equality. This iterative equation clearly has a unique solution.  
\end{proof}
It is clear that our solution $I = \sum_{n\geq 2}I_n$ of the Fedosov equation \eqref{eqn:Fedosov-equation-Wick} is exactly the unique one satisfying the conditions \eqref{equation: holomorphic-gauge-fixing} and \eqref{equation: normalization-condition}.

There were a number of works on the Fedosov construction of Wick type deformation quantizations on (pseudo-)K\"ahler manifolds \cites{Karabegov00, Bordemann, Neumaier}. Notice that, in all these works, the authors were using the same gauge condition, namely, Fedosov's original condition $\delta^{-1}I=0$, when solving the Fedosov equation. 
For the purpose of deformation quantization, there is no essential difference between these two choices of gauge conditions.
However, here are some interesting features of our construction which were not found in previous ones:
\begin{enumerate}
	\item As we have emphasized, our Fedosov connection is a quantization of Kapranov's $L_\infty$-algebra structure on a K\"ahler manifold.
	\item In our construction, the sections $O_f$ corresponding to holomorphic functions $f$ do not receive any quantum corrections. This is consistent with the Berezin-Toeplitz quantization, and also the local picture where $z$ acts as the {\em creation operator} which is classical while $\bar{z}$ acts as the {\em annihilation operator} $\hbar\frac{\partial}{\partial z}$ which is quantum.
	\item Our quantization is in a certain sense ``polarized'': only half of the functions, i.e., the anti-holomorphic ones receive quantum corrections.
\end{enumerate}

\subsubsection{The Karabegov form}\label{section: Karabegov-form}
\

In \cite{Karabegov96}, Karabegov gives a complete classification of deformation quantizations of Wick type on a K\"ahler manifold:\footnote{Actually the roles of the holomorphic and anti-holomorphic variables in \cite{Karabegov96} were reversed, but this does not affect the results.}
\begin{thm}[Theorem 2 in \cite{Karabegov96}]\label{theorem: Karabegov-classification-Wick-star-product}
Deformation quantizations of Wick type  (or with separation of variables) on a K\"ahler manifold $X$ are in one-to-one correspondence with formal deformations of the K\"ahler metric $\omega$ on $X$. 
\end{thm}
To see how a formal deformation of $\omega$ is constructed from a Wick type star product, let us recall the following proposition in \cite{Karabegov96}:
\begin{prop}[Proposition 1 in \cite{Karabegov96}]\label{proposition: function-u-k}
	Let $X$ be a K\"ahler manifold with K\"ahler form $\omega$ and $\star$ be a formal star product with separation of variables. Then, on each contractible coordinate chart $U\subset X$ with any holomorphic coordinate system $(z^1,\cdots, z^n)$, there exist formal functions $u_1,\cdots, u_m\in C^\infty(U)[[\hbar]]$ such that $[u_k, z^{k'}]_\star = \hbar\delta_{kk'} $.
\end{prop}
This proposition gives a locally defined formal differential form $\bar{\partial}(-u_kdz^k)$ of type $(1,1)$ on each chart $U$, which patch together to a globally defined closed formal differential form, called the {\em Karabegov form} associated to $\star$. Moreover, this formal $(1,1)$-form is a deformation of $\omega$, i.e., of the form $\omega+O(\hbar)$.
\begin{rmk}
In \cite{Karabegov96}, the Karabegov form is defined as $\sqrt{-1}\cdot\bar{\partial}(-u_kdz^k)$. We are using a different normalization since our star products satisfy the condition that $C_1(f,g)-C_1(g,f)=\{f,g\}$, instead of $C_1(f,g)-C_1(g,f)=\sqrt{-1}\cdot\{f,g\}$.
\end{rmk}

Now let $\alpha=\sum_{i\geq 1}\hbar^i\omega_i$ be a closed formal differential form of type $(1,1)$, and $D_{F,\alpha}$ and $\star_\alpha$ be respectively the associated Fedosov abelian connection and Wick type star product constructed in Theorem \ref{proposition: Fedosov-connection-general}. To calculate the Karabegov form associated to $\star_{\alpha}$, we begin with a lemma.
\begin{lem}\label{lemma: function-u-k}
Let $\alpha=\sum_{i\geq 1}\hbar^i\omega_i$ be a closed formal differential form of type $(1,1)$ and $-\sum_{i\geq 1}\hbar^i\rho_i$ be a potential of $\alpha$ (i.e., $\bar{\partial}\partial\rho_i=\omega_i$), and let $\rho$ be a potential of $\omega$. For the (locally defined) formal functions
$$
u_k=\frac{\partial}{\partial z^k}\left(\rho + \sum_{i\geq 1}\hbar^i \rho_i\right),
$$
the terms in $O_{u_k}$ which live in $\overline{\mathcal{W}}_X$ (which we call ``terms of purely anti-holomorphic type'') are given by  
$$
u_k+\omega_{k\bar{m}}\bar{y}^m.
$$
\end{lem}
\begin{proof}
Recall that  $O_{u_k}$ is the unique solution of the iterative equation:
\begin{align*}
O_{u_k} = u_k+\delta^{-1}\circ\left(\nabla+\frac{1}{\hbar}[I_\alpha,-]_{\star}\right)(O_{u_k}).
\end{align*}
Observe that if a monomial $A$ does not live in $\A_X^{\bullet}(\overline{\mathcal{W}}_{X})$, then $\nabla A+\frac{1}{\hbar}[I_\alpha,A]_{\star}$ does not have terms living in $\A_X^{\bullet}(\overline{\mathcal{W}}_{X})$.  So we can prove the theorem by an induction on the weights of `` terms of purely anti-holomorphic type'' in $O_{u_k}$.

The terms in $O_{u_k}$ of weight $1$ are given by
$$
\frac{\partial^2\rho}{\partial\bar{z}^l\partial z^k}\bar{y}^l=\omega_{k\bar{l}}\bar{y}^l.
$$
We know from the iterative equation that the weight $2$ terms are given by
$$
\delta^{-1}\circ\nabla^{0,1}(\omega_{k\bar{l}}\bar{y}^l),
$$
which vanish since the Levi-Civita connection is compatible with both the symplectic form and the complex structure. The next terms are 
\begin{align*}
&\delta^{-1}\left(\nabla^{0,1}\left(\hbar\frac{\partial\rho_1}{\partial z^k}\right)+\frac{1}{\hbar}\left[\hbar\frac{\partial^2\rho_1}{\partial\bar{z}^n\partial z^m}d\bar{z}^n\otimes y^m,\omega_{k\bar{l}}\bar{y}^l\right]_{\star}\right)\\
=&\delta^{-1}\left(\hbar\frac{\partial^2\rho_1}{\partial z^k\partial\bar{z}^l}d\bar{z}^l+\hbar\frac{\partial^2\rho_1}{\partial\bar{z}^n\partial z^m}d\bar{z}^n\omega_{k\bar{l}}\omega^{m\bar{l}}\right)=0.
\end{align*}
Thus the weight $3$ terms in $O_{u_k}$ of purely anti-holomorphic type vanish. This argument can be generalized to all such terms of higher weights. 
\end{proof}

\begin{thm}\label{theorem: Karabegov-form-Fedosov-deformation-quantization}
For every closed formal differential form $\alpha$ of type $(1,1)$, the star product $\star_{\alpha}$ defined in \eqref{equation: formal-star-product} has Karabegov form $\omega-\alpha$. 
\end{thm}
\begin{proof}
  Let $U$ be any contractible coordinate chart in $X$, with local holomorphic coordinates $(z^1,\cdots, z^k)$. We define the functions $u_k$ as in Lemma \ref{lemma: function-u-k}. Then the flat section $O_{z^k}$ can be explicitly written as $O_{z^k}=z^k+y^k+\cdots$, where all the terms live in $\mathcal{W}_X$. From the definition of the fiberwise Wick product on $\mathcal{W}_{X,\C}$ and that of the symbol map, we only need those terms in $O_{u_k}$ which are of ``purely anti-holomorphic type'' for the following computation: 
 	$$O_{u_k}\star O_{z^{k'}}-O_{z^{k'}}\star O_{u_k}
  = \omega_{k\bar{m}}\bar{y}^m\star y^{k'}-y^{k'}\star (\omega_{k\bar{m}}\bar{y}^m)
  = -y^{k'}\star (\omega_{k\bar{m}}\bar{y}^m)
  = \hbar\delta_{kk'}.$$
 Thus we have shown that the functions $u_k$'s satisfy the condition in Proposition \ref{proposition: function-u-k}. According to its construction, the Karabegov form is then given by
 $$
 \bar{\partial}(-u_kdz^k)
  = -\frac{\partial u_k}{\partial\bar{z}^l}d\bar{z}^l\wedge dz^k
  = \left(\frac{\partial^2\rho}{\partial z^k\partial\bar{z}^l}+\sum_{i\geq 1}\hbar^i\frac{\partial^2\rho_i}{\partial z^k\partial\bar{z}^l}\right)dz^k\wedge d\bar{z}^l
  = \omega-\alpha.
 $$
\end{proof}

Combining Theorems \ref{theorem: Karabegov-classification-Wick-star-product} and \ref{theorem: Karabegov-form-Fedosov-deformation-quantization}, we see that any star product of Wick type on a K\"ahler manifold arises from our construction:
\begin{cor}\label{corollary:Wick-type}
On a K\"ahler manifold $X$, any deformation quantization of Wick type is of the form $\star_{\alpha}$ for some closed formal $(1,1)$ form $\alpha$.
\end{cor}

\section{From Fedosov to Batalin-Vilkovisky}

Let us first recall the definition of traces of a star product:
\begin{defn}
	Let $(C^\infty(X)[[\hbar]],\star)$ denote a deformation quantization of a symplectic manifold $X$ of dimension $2n$. A {\em trace} of the star product $\star$ is a linear map $\Tr:C^\infty(X)[[\hbar]]\rightarrow\C[[\hbar]]/\hbar^n$ such that 
	\begin{enumerate}
		\item $\Tr(f\star g)=\Tr(g\star f)$;
		\item $\hbar^n\cdot\Tr(f)=\int_X f\cdot\omega^n+O(\hbar)$.
	\end{enumerate}
	In particular, $\Tr(1)$ is called the {\em algebraic index} of $\star$. 
\end{defn}
From the point of view of quantum field theory (QFT), traces are defined by correlation functions of local quantum observables. The Fedosov quantization describes the local data of a quantum mechanical system, namely, the cochain complex 
$$
(\A_X^\bullet(\mathcal{W}_{X,\C}), D_{F,\alpha})
$$
gives the {\em cochain complex of local quantum observables} of a sigma model from $S^1$ to the target $X$. To get global quantum observables and define the correlation functions properly, we study the Batalin-Vilkovisky (BV) quantization \cite{BV} of this quantum mechanical system, from which we would obtain a factorization map from local to global quantum observables.
For a detailed explanation of the physical background, we refer the readers to \cite{GLL}. 

Mathematically, a BV quantization can be formulated as a solution of the quantum master equation (QME). Our main result in this section is that, the canonical solution of the QME associated to the Fedosov abelian connection $D_{F,\alpha}$ is one-loop exact. This leads to a very neat cochain level formula for the algebraic index.



The organization of this section is as follows:
Section \ref{section: BV-integration} is a review of the construction of BV quantization from Fedosov abelian connections. We mainly follow the treatment in \cite[Sections 2.3-2.5]{GLL} but there is one key difference, namely, the choice of the propagtor, in order to reflect the K\"ahlerian condition.
In Section \ref{section: computation-Feynman-graphs}, we prove the main result of this section saying that our solutions of the QME are all one-loop exact, and we explain how this can lead to the cochain level formula for the trace.



\subsection{BV quantization}\label{section: BV-integration}
\

This subsection is largely a review of the construction of BV quantizations from Fedosov quantizations in \cite[Sections 2.3-2.5]{GLL}. The main difference lies in the choice of the propagator -- we will give a construction of the so-called {\em polarized propagator}, which is more compatible with the K\"ahlerian condition and leads to some special features of the resulting BV quantizations.

\subsubsection{Geometry of BV bundles and the QME}
\

The cochain complex of global quantum observables can be described in a differential geometric way. We start with the BV bundle on $X$:
\begin{defn}[cf. Definition 2.19 in \cite{GLL}]\label{definition: BV-bundle}
 The {\em BV bundle} of a K\"ahler manifold $X$ is defined to be 
 $$
 \widehat{\Omega}^{-\bullet}_{TX}:=\widehat{\Sym}(T^*X_{\C})\otimes\wedge^{-\bullet}(T^*X_{\C}),\quad \wedge^{-\bullet}(T^*X_{\C}):=\bigoplus_k\wedge^k(T^*X_{\C})[k],
 $$
 where $\wedge^k(T^*X_{\C})$ has cohomological degree $-k$.
\end{defn}

For any tensor power of the BV bundle, we have the canonically defined {\em multiplication}:
$$
\text{Mult}:(\widehat{\Omega}^{-\bullet}_{TX})^{\otimes k} := \widehat{\Omega}^{-\bullet}_{TX}\otimes\cdots\otimes\widehat{\Omega}^{-\bullet}_{TX}\rightarrow\widehat{\Omega}^{-\bullet}_{TX},
$$
which can be extended $\A_X^{\bullet}$-linearly to $\text{Mult}: \A_X^{\bullet}(\widehat{\Omega}^{-\bullet}_{TX})^{\otimes k}\rightarrow \A_X^{\bullet}(\widehat{\Omega}^{-\bullet}_{TX})$.
To describe the differential on the BV bundle, we consider the fiberwise de Rham operator
$d_{TX}:\widehat{\Omega}_{TX}^{-\bullet}\rightarrow\widehat{\Omega}^{-(\bullet+1)}_{TX}$,
and the contraction
$\iota_{\Pi}:\widehat{\Omega}_{TX}^{-\bullet}\rightarrow\widehat{\Omega}^{-(\bullet-2)}_{TX}$
by the Poisson tensor. We also have similarly defined operators $\partial_{TX}, \bar{\partial}_{TX}$.
There is also the {\em BV operator} defined by 
$$
\Delta:=[d_{TX},\iota_{\Pi}].
$$
The operators $d_{TX},\iota_{\Pi}$ and $\Delta$ all extend $\A_X^{\bullet}$-linearly to operators on $\A_X^{\bullet}(\widehat{\Omega}^{-\bullet}_{TX})^{\otimes k}$.
The failure of the BV operator $\Delta$ being a derivation is known as the {\em BV bracket}:
\begin{equation*}
\{A,B\}_\Delta:=\Delta(A\cdot B)-\Delta(A)\cdot B\pm A\cdot\Delta(B).
\end{equation*}
It is clear that the operators $d_{TX}, \iota_{\Pi}$ and $\Delta$ all commute with the multiplication map $\text{Mult}$. We also have $[\Delta, \nabla]=0$ since $\nabla(\omega)=0$, and $\Delta^2=0$ by the Jacobi identity for the Poisson tensor.

\begin{lem}[Lemma 2.21 in \cite{GLL}]\label{lemma: BV-operator-differential}
The operator 
$$Q_{BV}:=\nabla+\hbar\Delta+\frac{1}{\hbar}d_{TX}R_{\nabla}$$
is a differential on the BV bundle (i.e., $Q_{BV}^2=0$), which we call the {\em BV differential}. 
\end{lem}

\begin{defn}[Definition 2.22 in \cite{GLL}]\label{definition: QME}
A section $r$ of the BV bundle is said to {\em satisfy the quantum master equation (QME)} if
\begin{equation}\label{equation: quantum-master-equation}
 Q_{BV}(e^{r/\hbar})=0.
\end{equation}
\end{defn}
This is equivalent to $\nabla r+\hbar\Delta r+\{r,r\}_{\Delta}+d_{TX}R_{\nabla}=0$ since
\begin{align*}
Q_{BV}(e^{r/\hbar}) =\left(\nabla+\hbar\Delta+\frac{1}{\hbar}d_{TX}R_{\nabla}\right)(e^{r/\hbar})
=\frac{1}{\hbar}\left(\nabla r+\hbar\Delta r+\{r,r\}_{\Delta}+d_{TX}R_{\nabla}\right)\cdot e^{r/\hbar}.
\end{align*}

\begin{lem}[Lemma 2.23 in \cite{GLL}]\label{lemma: global-quantum-observable}
Given a solution $\gamma_\infty$ of the QME \eqref{equation: quantum-master-equation}, the operator 
\begin{equation}\label{equation: differential-global-quantum-observable}
\nabla + \hbar\Delta + \{\gamma_\infty,-\}_\Delta
\end{equation}
is a differential on the BV bundle.
The {\em cochain complex of global quantum observables} is defined as $(\A_X^\bullet\left(\hat{\Omega}_{TX}^{-\bullet}\right)[[\hbar]], \nabla  + \hbar\Delta + \{\gamma_\infty,-\}_\Delta)$. 
\end{lem}

\begin{lem}[Lemma 2.24 in \cite{GLL}]
	The {\em fiberwise Berezin integration}, defined by taking the top degree component in odd variables and setting the even variables to $0$:
	\begin{equation*}\label{equation: Berezin-integration}
	\int_{Ber}:\A_X^\bullet\left(\hat{\Omega}_{TX}^{-\bullet}\right) \rightarrow \A_X^\bullet,
	\qquad a \mapsto\frac{1}{n!}(\iota_{\Pi})^n(a)\bigg|_{y^i=\bar{y}^j=0},
	\end{equation*}
	is a cochain map, with respect to the BV differential $Q_{BV}$ on $\A_X^\bullet\left(\hat{\Omega}_{TX}^{-\bullet}\right)$ and the de Rham differential on $\A_X^\bullet$.   
\end{lem}
From this lemma we get a well-defined composition map on cohomology classes:
\begin{equation*}\label{equation: BV-integration-composite-ordinary-integration}
H^*(\A_X^\bullet\left(\hat{\Omega}_{TX}^{-\bullet}\right)[[\hbar]])\overset{\int_{Ber}}{\longrightarrow} H^*_{dR}(X)[[\hbar]]\overset{\int_{X}}{\longrightarrow}\mathbb{C}[[\hbar]].
\end{equation*}
We can thus define the correlation functions (or expectation values) of global quantum observables:
\begin{defn}\label{definition: correlation-function}
Let $\gamma_\infty$ be a solution of the QME \eqref{equation: quantum-master-equation} and let $O$ be a global quantum observable. The {\em correlation function of $O$} is defined as
\begin{equation*}\label{equation: correlation-function}
  \langle O\rangle:=\int_X \int_{Ber} O \cdot e^{\gamma_\infty/\hbar}.
\end{equation*}
\end{defn}

\subsubsection{Configuration spaces and propagator}\label{section: homotopy-group-flow}
\

To construct solutions of the QME and define the local-to-global factorization map, we need the {\em homotopy group flow operator} using the technique of configuration spaces. Our formulation here follows that in \cite[Section 2.4]{GLL} but with a significant modification of the definition of the propagator in order to adapt to the K\"ahler setting.

We give a brief introduction to configuration spaces, and refer to \cites{AS2,GLL} for more details on the construction and basic facts.  Let $M$ be a smooth manifold and let $V:=\{1,\cdots,n\}$ for some integer $n\geq 2$.  The compactified configuration space $M[V]$ is  a smooth manifold with corners as we now recall. 
Let $S\subset V$ be any subset with $|S|\geq 2$. We denote by $M^S$ the set of all maps from $S$ to $M$ and by $\Delta_S \subset M^S$ the small diagonal. The real oriented blow up of $M^S$ along $\Delta_S$, denoted as $\text{Bl}(M^S,\Delta_S)$, is a manifold with boundary whose interior is diffeomorphic to $M^S\setminus\Delta_S$ and whose boundary is diffeomorphic to the unit sphere bundle associated to the normal bundle of $\Delta_S$ inside $M^S$.

\begin{defn}\label{defn-configuration}
Let $V$ be as above and let $M_0^V$ be the configuration space of $n=|V|$ pairwise different points in $M$,
\[
M_0^V:=\{(x_1,\cdots, x_n)\in M^{V}: x_i\not=x_j\ \text{for}\ i\not=j\}.
\]
There is the following embedding:
\[
M_0^V\hookrightarrow M^V\times\prod_{|S|\geq 2}\text{Bl}(M^S,\Delta_S).
\]
The space $M[V]$ is defined as the closure of the above embedding. For $V=\{1,\cdots, n\}$, we will denote $M[V]$ by $M[n]$.
\end{defn}
It follows from the definition of $M[V]$ that for every subset $S\subset V$, there is a natural projection
\[
\pi_S: M[V]\rightarrow M[S].
\]
and
\[
\pi:  M[n]\to M^n. 
\]
\begin{eg}\label{example: two-point-configuration}
For $M=S^1$, the space $S^1[2]$ is a cylinder $S^1\times [0,1]$, which can be explicitly constructed by cutting along the diagonal in $S^1\times S^1$. Its boundary is given by $\partial S^1[2]=S^1\times S^0$, i.e.,  two copies of $S^1$.  More explicitly, let us fix a parametrization  
\[
  S^1=\{e^{2\pi i \theta}|0\leq \theta <1\}.
\]
Then $S^1[2]$ is parametrized by a cylinder 
\begin{equation}\label{equation: parametrization-S-2}
  S^1[2]=\{(e^{2\pi i\theta},u)| 0\leq \theta<1, 0\leq u\leq 1\}. 
\end{equation}
Further, 
\[
   \pi: S^1[2]\to (S^1)^2, \quad (e^{2\pi i \theta},u)\mapsto (e^{2\pi i (\theta+u)}, e^{2\pi i \theta}).
\]
\end{eg}
We now consider the lifting of propagators to compactifications of configuration spaces.
\begin{prop}[cf. Proposition B.4 in \cite{GLL}]\label{prop:propagator-step-function}
Let $S^1$ be the interval $[0,1]$ with $0$ and $1$ identified. In terms of the polar coordinates on the two copies of $S^1$, the propagator $P(\theta_1,\theta_2)$ is the following periodic  function of $\theta_1-\theta_2\in\R\backslash\mathbb{Z}$
\begin{equation}\label{eqn:propagator-formula} 
P(\theta_1, \theta_2)=\theta_1-\theta_2-\frac{1}{2},\quad \text{if}\hspace{2mm}  0< \theta_1-\theta_2<1. 
\end{equation}
In terms of the parametrization of $S^1[2]$ in equation \eqref{equation: parametrization-S-2}, we also write the propagator as
\begin{equation}\label{definition: propagator-no-polarization}
P(\theta,u) := u-\frac{1}{2}.
\end{equation}
\end{prop}

The function $P$ on $S^1[2]$ is called the {\em propagator} (see \cite[Definition 2.30]{GLL}) because it is the derivative of the Green's function on $S^1$ with respect to the standard flat metric, and thus represents the propagator of topological quantum mechanics on $S^1$ (see \cite[Remarks 2.31 and B.6]{GLL}).
When restricted to the open subset $S^1\times S^1\setminus\Delta \subset S^1[2]$, the propagator $P$ is anti-symmetric in the two copies of $S^1$: if $0<\theta_1-\theta_2<1$, then
$$
P(\theta_2,\theta_1)=\theta_2-\theta_1+1-\frac{1}{2}=-P(\theta_1,\theta_2).
$$

\subsubsection{Homotopy group flow operator}
\

We can now define the {\em polarized propagator} in the K\"ahler setting as a combination of $P$ and the K\"ahler form on $X$:
\begin{defn}[cf. Definition 2.32 in \cite{GLL}]\label{definition: propagator-polarized}
 We define the $\A_{S^1[k]}^{\bullet}$-linear operators 
 $$\partial_P, D: \A_{S^1[k]}^{\bullet}\otimes_{\mathbb{C}}\A_X^{\bullet}(\widehat{\Omega}_{TX}^{-\bullet})^{\otimes k}\rightarrow \A_{S^1[k]}^{\bullet}\otimes_{\mathbb{C}}\A_X^{\bullet}(\widehat{\Omega}_{TX}^{-\bullet})^{\otimes k}$$
 by
 \begin{enumerate}
 \item 
 $\partial_P(a_1\otimes\cdots\otimes a_k)$
 \begin{align*}
 \hspace{2cm}:=&\sum_{1\leq\alpha<\beta\leq k}\pi_{\alpha\beta}^*(P+\frac{1}{2})\otimes_{\mathbb{C}}(\omega^{i\bar{j}}(x)a_1\otimes\cdots\otimes\mathcal{L}_{\partial_{z^i}}a_\alpha\otimes\cdots\otimes\mathcal{L}_{\partial_{\bar{z}^j}}a_\beta\otimes\cdots\otimes a_k)\\
 &\ -\sum_{1\leq\alpha<\beta\leq k}\pi_{\alpha\beta}^*(P-\frac{1}{2})\otimes_{\mathbb{C}}(\omega^{i\bar{j}}(x)a_1\otimes\cdots\otimes\mathcal{L}_{\partial_{\bar{z}^j}}a_\alpha\otimes\cdots\otimes\mathcal{L}_{\partial_{z^i}}a_\beta\otimes\cdots\otimes a_k);
 \end{align*}
 \item $\displaystyle D(a_1\otimes\cdots\otimes a_k):=\sum_{1\leq i\leq k}\pm d\theta_{i}\otimes_{\mathbb{C}} (a_1\otimes\cdots\otimes d_{TX}a_i\otimes\cdots a_k)$.
 \end{enumerate}
Here $a_i \in \A_X^{\bullet}(\widehat{\Omega}_{TX}^{-\bullet})$, $\pi_{\alpha\beta}: S^1[k] \to S^1[2]$ is the forgetful map to the two points indexed by $\alpha, \beta$, $\theta_i \in [0,1)$ is the parameter on the $S^1$ indexed by $1 \leq i \leq k$ and $d\theta_i$ is a 1-form on $S^1[k]$ via the pullback $\pi_i:S^1[k] \to S^1$, and finally $\pm$ are appropriate Koszul signs.
\end{defn}
\begin{rmk}
 In the definition of the operator $\partial_P$, the sum over indices $\alpha<\beta$ is saying that $\partial_P$ applies to labeling of different vertices. This rule here follows from our later computation in Theorem 3.19 that homotopy group flow operator without tadpole graphs turns solutions of Fedosov equation to solution of Quantum master equation.
\end{rmk}

We have the following decomposition of the operator $\partial_P$:
\begin{lem}\label{remark: propagator-comparision-with-symplectic-case}
The operator $\partial_P$ in Definition \ref{definition: propagator-polarized} can be written as the sum $\partial_{P}=\partial_{P_1}+\partial_{P_2}$, where  
\begin{align*}
  \partial_{P_1}(a_1&\otimes\cdots\otimes a_k)=\sum_{1\leq\alpha<\beta\leq k}\pi_{\alpha\beta}^*(P)\otimes_{\mathbb{C}}(\omega^{i\bar{j}}\cdot a_1\otimes\cdots\otimes\mathcal{L}_{\partial_{z^i}}a_\alpha\otimes\cdots\otimes\mathcal{L}_{\partial_{\bar{z}^j}}a_\beta\otimes\cdots\otimes a_k)\\
  &-\sum_{1\leq\alpha<\beta\leq k}\pi_{\alpha\beta}^*(P)\otimes_{\mathbb{C}}(\omega^{i\bar{j}}\cdot a_1\otimes\cdots\otimes\mathcal{L}_{\partial_{\bar{z}^j}}a_\alpha\otimes\cdots\otimes\mathcal{L}_{\partial_{z^i}}a_\beta\otimes\cdots\otimes a_k),
\end{align*}
and
\begin{align*}
  \partial_{P_2}(a_1&\otimes\cdots\otimes a_k)=\frac{1}{2}\sum_{1\leq\alpha<\beta\leq k}(\omega^{i\bar{j}}\cdot a_1\otimes\cdots\otimes\mathcal{L}_{\partial_{z^i}}a_\alpha\otimes\cdots\otimes\mathcal{L}_{\partial_{\bar{z}^j}}a_\beta\otimes\cdots\otimes a_k)\\
  &+\frac{1}{2}\sum_{1\leq\alpha<\beta\leq k}(\omega^{i\bar{j}}\cdot a_1\otimes\cdots\otimes\mathcal{L}_{\partial_{\bar{z}^j}}a_\alpha\otimes\cdots\otimes\mathcal{L}_{\partial_{z^i}}a_\beta\otimes\cdots\otimes a_k).
\end{align*}
\end{lem}
Notice that both $\partial_{P_1}$ and $\partial_{P_2}$ satisfy the property that they maps symmetric tensors to symmetric ones.  And that $\partial_{P_1}$ coincides with the propagator used in \cite{GLL}. 

\begin{lem}[cf. Lemma 2.33 in \cite{GLL}]\label{lemma: differential-propogator-equal-BV}
 As operators on the BV bundle, we have
 \begin{align*}
 [d_{S^1}, \partial_P]=[\Delta, D], \qquad D^2=0,\qquad
 [\partial_P, D]=0.
 \end{align*}
\end{lem}



Lemma \ref{lemma: differential-propogator-equal-BV} says that the operators $\hbar\partial_P$ and $D$ commute, so we can formally define the following operator on the BV bundle:
\begin{equation*}\label{equation: homotopy-group-flow-operator}
e^{\hbar\partial_P+D} := \sum_{k\geq 0}\frac{1}{k!}(\hbar\partial_P+D)^k.
\end{equation*}
Here is the definition of the {\em homotopy group flow operator} on the BV bundle:
\begin{defn}[cf. Definition 2.34 in \cite{GLL}]\label{definition: gamma-infty}
 Given any $\gamma\in\A_X^\bullet(\mathcal{W}_{X,\C})$, we define $\gamma_\infty\in\A_X^\bullet(\hat{\Omega}_{TX}^{-\bullet})[[\hbar]]$ by 
 $$
 e^{\gamma_{\infty}/\hbar}:=\text{Mult}\int_{S^1[*]}e^{\hbar\partial_P+D}e^{\otimes\gamma/\hbar},
 $$
 where $e^{\otimes\gamma/\hbar}:=\sum_{k\geq 0}\frac{1}{k!\hbar^k}\gamma^{\otimes k}, \gamma^{\otimes k}\in\A_X^\bullet(\hat{\Omega}_{TX}^{-\bullet})^{\otimes k})$, and $\int_{S^1[k]}:\A_{S^1[k]}^{\bullet}\otimes_{\mathbb{C}}\A_X^{\bullet}(\widehat{\Omega}_{TX}^{-\bullet})^{\otimes k}\rightarrow \A_X^{\bullet}(\widehat{\Omega}_{TX}^{-\bullet})^{\otimes k}$ is the integration $\int_{S^1[k]}$ over the  configuration space $S^1[k]$.
\end{defn}
\begin{notn}
 We will use the notation $\int_{S^1[*]}$ to denote the sum of the above integral operators $\int_{S^1[k]}$ for mixed type tensors. 
\end{notn}

\begin{lem}[cf. Lemma 2.37 in \cite{GLL}]\label{lemma: BV-action-boundary-configuration-space}
For any $\gamma\in\A_X^\bullet(\mathcal{W}_{X,\C})$, we have
 \begin{equation}\label{equation: BV-operator-equal-boundary-integral}
 \hbar\Delta e^{\gamma_\infty/\hbar}=\text{Mult}\int_{S^1[*]}e^{\hbar\partial_P+D}\frac{1}{\hbar}(\gamma\star\gamma)\otimes e^{\otimes\gamma/\hbar}. 
 \end{equation}
\end{lem}
\begin{proof}
The proof is very similar to that of \cite[Lemma 2.37]{GLL}, except that here we use the polarized propagator $\partial_P$ instead of the propagator $\partial_{P_1}$ in \cite{GLL}.
	
Using the commutation relations in Lemma \ref{lemma: differential-propogator-equal-BV}, we have
\begin{align*}
 \hbar\Delta e^{\gamma_\infty/\hbar}&=\text{Mult}\int_{S^1[*]}\hbar\Delta e^{\hbar\partial_P+D}e^{\otimes\gamma/\hbar}\\
 &=\text{Mult}\int_{S^1[*]}d_{S^1}e^{\hbar\partial_P+D}e^{\otimes\gamma/\hbar}+\text{Mult}\int_{S^1[*]}(\hbar\Delta-d_{S^1})e^{\hbar\partial_P+D}e^{\otimes\gamma/\hbar}\\
 &=\text{Mult}\int_{\partial S^1[*]}e^{\hbar\partial_P+D}e^{\otimes\gamma/\hbar}+\text{Mult}\int_{S^1[*]}e^{\hbar\partial_P+D}(\hbar\Delta-d_{S^1})e^{\otimes\gamma/\hbar}\\
 &=\text{Mult}\int_{\partial S^1[*]}e^{\hbar\partial_P+D}e^{\otimes\gamma/\hbar},
\end{align*}
where in the last step we used the fact that $\hbar\Delta \left(e^{\otimes\gamma/\hbar}\right)=d_{S^1}\left(e^{\otimes\gamma/\hbar}\right)=0$ (here $d_{S^1}$ annihilates $e^{\otimes\gamma/\hbar}$ because $e^{\otimes\gamma/\hbar}$ does not depend on the configuration space $S^1[*]$ and $\Delta$ annihilates $e^{\otimes\gamma/\hbar}$ by type reasons).

We now consider the term $\text{Mult}\int_{\partial S^1[*]}e^{\hbar\partial_P+D}e^{\otimes\gamma/\hbar}$. Here we need to apply the stratification of the compactified configuration spaces, and refer to \cite{GLL} for more details. For $M=S^1$, there is an explicit description of the boundary (i.e., codimension $1$ strata) of the configuration space:
$$
\partial S^1[k]=\bigcup_{I\subset\{1,\cdots,k\}, |I|\geq 2}\pi^{-1}(D_I),
$$
where $D_I\subset(S^1)^k$ is the small diagonal where points in those $S^1$'s indexed by $I$ coincide (see e.g. \cite[Appendix B]{GLL} for more details). Similar to the real symplectic case, only components corresponding to those indices with $|I|=2$ will contribute non-trivially to the boundary integral.  
Every such index $I$ corresponds to a pair $\alpha<\beta\in\{1,\cdots,k\}$. The corresponding boundary strata has two components, each isomorphic to $S^1[k-1]$. So we have
\begin{align*}
\text{Mult}\int_{\partial S^1[k]}e^{\hbar\partial_P+D}e^{\otimes\gamma/\hbar}&=\text{Mult}\int_{S^1[k-1]}e^{\hbar\partial_P+D}\frac{1}{2\hbar}[\gamma,\gamma]_{\star}\otimes e^{\otimes\gamma/\hbar}\\
&=\text{Mult}\int_{S^1[k-1]}e^{\hbar\partial_P+D}\frac{1}{\hbar}(\gamma\star\gamma)\otimes e^{\otimes\gamma/\hbar}.
\end{align*}
By adding the above equalities for all $k\geq 2$, we obtain equation \eqref{equation: BV-operator-equal-boundary-integral}.  
We emphasize that the polarized propagator plays a key role here: the operators on the two connected components precisely correspond to $A\star B$ and $B\star A$, and this explains why the bracket $[-,-]_\star$ shows up. 
\end{proof}

\subsubsection{Solutions of the QME from Fedosov abelian connections}\label{section: Fedosov-to-QME}
\

From the perspective of BV quantization, a solution of the QME \eqref{equation: quantum-master-equation} is the $\infty$-scale effective interaction of the quantum mechanical system on $X$. Such a solution can be constructed by running the homotopy group flow.

We first consider the following section of the BV bundle:
\begin{equation*}\label{equation: curvature-BV-bundle}
 \tilde{R}_\nabla := (\partial_{TX}\circ\bar{\partial}_{TX})(R_\nabla)=\sqrt{-1}\cdot R_{i\bar{j}k\bar{l}}dz^i\wedge d\bar{z}^j\otimes dy^{k}\wedge d\bar{y}^{l}.
\end{equation*}
\begin{lem}\label{lemma: tilde-R-properties}
We have $\Delta(\tilde{R}_\nabla) = \nabla(\tilde{R}_\nabla) = 0$. 
\end{lem}
\begin{proof}
The vanishing of the first term follows from the type of the BV operator $\Delta$, and the second one is the Bianchi identity. 
\end{proof}

\begin{lem}\label{lemma: partial-P-equal-BV-bracket}
We have $\partial_{P_2}(d_{TX}R_\nabla\otimes\gamma)=-\frac{1}{2}\{\tilde{R}_\nabla,\gamma\}_{\Delta}$.
\end{lem}
\begin{proof}
This is a straightforward computation. For the LHS, we have
\begin{align*}
&\partial_{P_2}(d_{TX}R_\nabla\otimes\gamma)\\
=&\frac{1}{2}\omega^{p\bar{q}}\left(\mathcal{L}_{\partial_{z^p}}(d_{TX}R_\nabla)\otimes\mathcal{L}_{\partial_{\bar{z}^q}}\gamma+\mathcal{L}_{\partial_{\bar{z}^q}}(d_{TX}R_\nabla)\otimes\mathcal{L}_{\partial_{z^p}}\gamma\right) \\
=&\frac{\sqrt{-1}}{2}\omega^{p\bar{q}}\left(\mathcal{L}_{\partial_{z^p}}(R_{i\bar{j}k\bar{l}}dz^i\wedge d\bar{z}^j\otimes y^kd\bar{y}^l)\otimes\mathcal{L}_{\partial_{\bar{z}^q}}\gamma+\mathcal{L}_{\partial_{\bar{z}^q}}(R_{i\bar{j}k\bar{l}}dz^i\wedge d\bar{z}^j\otimes \bar{y}^ldy^k)\otimes\mathcal{L}_{\partial_{z^p}}\gamma\right)\\
=&\frac{\sqrt{-1}}{2}\omega^{p\bar{q}}\left((R_{i\bar{j}p\bar{l}}dz^i\wedge d\bar{z}^j\otimes d\bar{y}^l)\otimes\mathcal{L}_{\partial_{\bar{z}^q}}\gamma+(R_{i\bar{j}k\bar{q}}dz^i\wedge d\bar{z}^j\otimes dy^k)\otimes\mathcal{L}_{\partial_{z^p}}\gamma\right).
\end{align*}
On the other hand, recall that 
$\Delta=\omega^{p\bar{q}}\left(\mathcal{L}_{\partial_{z^p}}\iota_{\partial_{\bar{z}^q}}-\mathcal{L}_{\partial_{\bar{z}^q}}\iota_{\partial_{z^p}}\right)$.
Hence the RHS is given by 
\begin{align*}
&\{\tilde{R}_\nabla,\gamma\}_\Delta\\
=&\sqrt{-1}\omega^{p\bar{q}}\left(\iota_{\partial_{\bar{z}^q}}(R_{i\bar{j}k\bar{l}}dz^i\wedge d\bar{z}^j\otimes dy^k\wedge d\bar{y}^l)\otimes\mathcal{L}_{\partial_{z^p}}(\gamma)-\iota_{\partial_{z^p}}(R_{i\bar{j}k\bar{l}}dz^i\wedge d\bar{z}^j\otimes dy^k\wedge d\bar{y}^l)\otimes \mathcal{L}_{\partial_{\bar{z}^q}}(\gamma)\right)\\
=&-\sqrt{-1}\omega^{p\bar{q}}\left((R_{i\bar{j}k\bar{q}}dz^i\wedge d\bar{z}^j\otimes dy^k)\otimes\mathcal{L}_{\partial_{z^p}}(\gamma)+(R_{i\bar{j}p\bar{l}}dz^i\wedge d\bar{z}^j\otimes d\bar{y}^l)\otimes \mathcal{L}_{\partial_{\bar{z}^q}}(\gamma)\right)
\end{align*}
\end{proof}

We are now ready to construct our solutions of the QME:
\begin{thm}[cf. Theorem 2.26 in \cite{GLL}]\label{theorem: Fedosov-connection-RG-flow-QME}
Suppose that $\gamma$ is a solution of the Fedosov equation \eqref{equation: Fedosov-equation-gamma} and let $\gamma_\infty$ be defined as in Definition \ref{definition: gamma-infty}.
Then $e^{\tilde{R}_\nabla/2\hbar}\cdot e^{\gamma_\infty/\hbar}$ is a solution of the QME \eqref{equation: quantum-master-equation}, i.e., $Q_{BV}\left(e^{\tilde{R}_\nabla/2\hbar}\cdot e^{\gamma_\infty/\hbar}\right)=0$.
\end{thm}
\begin{proof}
Recall that $Q_{BV}=\nabla+\hbar\Delta+\hbar^{-1}d_{TX}R_\nabla$. We compute each term in $Q_{BV}(e^{\gamma_\infty/\hbar})$: 
\begin{enumerate}
 \item The first term is given by:
 \begin{align*}
\nabla(e^{\gamma_\infty/\hbar}) & = \nabla\left(\text{Mult}\int_{S^1[*]}e^{\hbar\partial_P+D}e^{\otimes\gamma/\hbar}\right)
 = \text{Mult}\int_{S^1[*]}e^{\hbar\partial_P+D}\nabla(e^{\otimes\gamma/\hbar})\\
& = \text{Mult}\int_{S^1[*]}e^{\hbar\partial_P+D}(\frac{1}{\hbar}\nabla\gamma\otimes e^{\gamma/\hbar}),
\end{align*}
where we are using the relations
$[\nabla, \hbar\partial_P+D]=\left[\nabla, \int_{S^1[*]}\right]=0$.
\item Using Lemma \ref{lemma: BV-action-boundary-configuration-space}, the second term is given by 
\begin{align*}
\hbar\Delta(e^{\gamma_\infty/\hbar})=&\text{Mult}\int_{S^1[*]}e^{\hbar\partial_P+D}\frac{1}{\hbar}(\gamma\star\gamma)\otimes e^{\otimes\gamma/\hbar}\\
=&\text{Mult}\int_{S^1[*]}e^{\hbar\partial_P+D}\frac{1}{\hbar}(\gamma\star\gamma-\omega_\hbar)\otimes e^{\otimes\gamma/\hbar}
.
\end{align*}
Note that the $-\omega_\hbar$ term appears because $D(-\omega_\hbar)=0$ by type reasons. 
\item For the third term, we consider the following term
$$
\text{Mult}\int_{S^1[*]}e^{\hbar\partial_P+D}\frac{R_{\nabla}}{\hbar}\otimes e^{\otimes\gamma/\hbar}=\text{Mult}\int_{S^1[*]}e^{\hbar\partial_P+D}(\frac{1}{\hbar}d\theta d_{TX}R_{\nabla})\otimes e^{\otimes\gamma/\hbar} 
$$
Notice that $\hbar\partial_P$ can be applied to $d_{TX}R_\nabla$ once, which gives rise to the following difference:
\begin{align*}
&\text{Mult}\int_{S^1[*]}e^{\hbar\partial_P+D}(\frac{1}{\hbar}d\theta d_{TX}R_{\nabla})\otimes e^{\otimes\gamma/\hbar}- \text{Mult}\int_{S^1[*]}(\frac{1}{\hbar}d\theta d_{TX}R_{\nabla})\otimes e^{\hbar\partial_P+D}\left(e^{\otimes\gamma/\hbar}\right)\\
=&\text{Mult}\int_{S^1[*]}e^{\hbar\partial_P+D}(\frac{1}{\hbar}d\theta d_{TX}R_{\nabla})\otimes e^{\otimes\gamma/\hbar}- \text{Mult}\int_{S^1[*]}(\frac{1}{\hbar}d_{TX}R_{\nabla})\otimes e^{\hbar\partial_P+D} e^{\otimes\gamma/\hbar}\\
=&\text{Mult}\int_{S^1[*]}e^{\hbar\partial_P+D}d\theta_1\cdot\frac{1}{\hbar}\cdot(-\frac{1}{2})\cdot\{\tilde{R}_{\nabla},d_{TX}\gamma\}_{\Delta}\otimes e^{\otimes\gamma/\hbar}\\
=&-\frac{1}{2\hbar}\{\tilde{R}_\nabla, \gamma_\infty\}_\Delta\cdot e^{\gamma_\infty/\hbar},
\end{align*}
where we have used Lemma \ref{lemma: partial-P-equal-BV-bracket} and the relations
$\left[\{\tilde{R}_{\nabla},-\}_\Delta,\hbar\partial_P\right]=\left[\{\tilde{R}_{\nabla},-\}_\Delta, D\right]=0$ in the last equality.
\end{enumerate}
By the above computations and Lemma \ref{lemma: tilde-R-properties}, we obtain that
\begin{align*}
 &Q_{BV}\left(e^{\tilde{R}_\nabla/2\hbar}\cdot e^{\gamma_\infty/\hbar}\right)\\
 =&e^{\tilde{R}_\nabla/2\hbar}\cdot Q_{BV}(e^{\gamma_\infty/\hbar})+\left((\nabla+\hbar\Delta)e^{\tilde{R}_\nabla/2\hbar}\right)\cdot e^{\gamma_\infty/\hbar}+\frac{1}{2\hbar}\{\tilde{R}_\nabla,\gamma_\infty\}_\Delta\cdot\left(e^{\tilde{R}_\nabla/2\hbar}\cdot e^{\gamma_\infty/\hbar}\right)\\
 =&e^{\tilde{R}_\nabla/2\hbar}\cdot Q_{BV}(e^{\gamma_\infty/\hbar})+\frac{1}{2\hbar}\{\tilde{R}_\nabla,\gamma_\infty\}_\Delta\cdot\left(e^{\tilde{R}_\nabla/2\hbar}\cdot e^{\gamma_\infty/\hbar}\right)\\
 =&\frac{1}{\hbar}\cdot e^{\tilde{R}_\nabla/2\hbar}\cdot\text{Mult}\int_{S^1[*]}e^{\hbar\partial_P+D}\left(\nabla\gamma+\frac{1}{\hbar}\gamma\star\gamma-\omega_\hbar+R_\nabla\right)\otimes e^{\otimes\gamma/\hbar}\\
 &\hspace{5mm}-\left(\frac{1}{2}\{\tilde{R}_\nabla,\gamma_\infty\}_\Delta+\frac{1}{2}\{\tilde{R}_\nabla,\gamma_\infty\}_\Delta\right)\cdot\left(e^{\tilde{R}_\nabla/2\hbar}\cdot e^{\gamma_\infty/\hbar}\right)\\
 =&\frac{1}{\hbar}\cdot e^{\tilde{R}_\nabla/2\hbar}\cdot\text{Mult}\int_{S^1[*]}e^{\hbar\partial_P+D}\left(\nabla\gamma+\frac{1}{\hbar}\gamma\star\gamma-\omega_\hbar+R_\nabla\right)\otimes e^{\otimes\gamma/\hbar}=0.
\end{align*}
\end{proof}

\begin{rmk}
	Compared with \cite[Theorem 2.26]{GLL}, our QME solution has the additional factor $e^{\tilde{R}_\nabla/2\hbar}$ due to our different choice of the propagator. We will see in the next subsection that this leads to nontrivial contribution from the tadpole graph in computing the partition function, which we do not see in the general symplectic case in \cite{GLL}.
\end{rmk}
We are now ready to define the local-to-global factorization map of quantum observables:
\begin{thm}[cf. Theorem 2.40 in \cite{GLL}]\label{theorem: local-to-global-factorization-map}
The {\em factorization map} defined by
\begin{align*}
[-]_\infty: \A_X^{\bullet}(\mathcal{W}_{X,\C})&\rightarrow\A_X^\bullet(\hat{\Omega}_{TX}^{-\bullet})[[\hbar]]\\
O&\mapsto [O]_\infty:=e^{-\gamma_\infty/\hbar}\cdot \left(\text{Mult}\int_{S^1[*]}e^{\hbar\partial_P+D}(O d\theta_1\otimes e^{\otimes\gamma/\hbar})\right)
\end{align*}
is a cochain map from the complex $(\A_X^{\bullet}(\mathcal{W}_{X,\C}), D_{F,\alpha})$ of local quantum observables to the complex $(\A_X^\bullet(\hat{\Omega}_{TX}^{-\bullet})[[\hbar]], \nabla +\{\gamma_\infty,-\}_\Delta+\hbar\Delta)$ of global quantum observables.
\end{thm}
\begin{cor}
Let $f\in C^\infty(X)[[\hbar]]$ be any smooth function and $O_f$ be the corresponding flat section under the Fedosov connection. Then $Q_{BV}\left([O_f]_\infty\cdot e^{\tilde{R}_\nabla/2\hbar}\cdot e^{\gamma_\infty/\hbar}\right)=0$. 
\end{cor}

\subsection{The trace}\label{subsection: trace-star-products}
\

For the above BV quantization, we can define the correlation function of a local quantum observable:
\begin{defn}\label{definition:correlation-functions}
 The {\em correlation function} of a formal smooth function $f\in C^\infty(X)[[\hbar]]$ is defined by
$$
\langle f\rangle := \int_X \int_{Ber} [O_f]_\infty\cdot e^{\tilde{R}_\nabla/2\hbar}\cdot e^{\gamma_\infty/\hbar}.
$$
\end{defn}
The proofs of the following propositions, which we omit, are the same as that in \cite{GLL}:
\begin{prop}[cf. Proposition 2.43 in \cite{GLL}]\label{proposition: correlation-function-vanish-on-commutators}
Let $f,g\in C^\infty(X)[[\hbar]]$ be two smooth functions on $X$, and let $h=[f,g]_{\star_\alpha}$ denote the commutator of $f,g$ under the Fedosov star product $\star_\alpha$. Then $[O_h]_\infty$ is exact under the differential \eqref{equation: differential-global-quantum-observable} and hence $\langle f*_\alpha g\rangle=\langle g*_\alpha f\rangle$.
\end{prop}

\begin{prop}[cf. Proposition 2.44 in \cite{GLL}]\label{proposition: leading-term-correlation-function}
The correlation function of $f\in C^\infty(X)[[\hbar]]$ is of the form
$$
\langle f\rangle=\int_X f\cdot\omega^n+O(\hbar).
$$
\end{prop}

\begin{cor}\label{corollary:trace}
	The association $\Tr: f \mapsto \langle f\rangle$ is the trace of the Fedosov star product $\star_\alpha$.
\end{cor}

To summarize, the BV quantization method gives rise to a way to explicitly compute an integral density of the trace of a Wick type star product $\star_\alpha$: for any smooth function $f\in C^\infty(X)$, we
\begin{enumerate}
 \item compute $O_f$ using the iterative equation \eqref{equation: iteration-equation-quantum-flat-section},
 \item obtain $\gamma_\infty$ as a Feynman graph expansion (see Lemma \ref{lemma: gamma-infty-graph-expression} below),
 \item compute $[O_f]_\infty$ via Theorem \ref{theorem: local-to-global-factorization-map}, and
 \item finally take the Berezin integral of $[O_f]_\infty\cdot e^{\tilde{R}_\nabla/2\hbar}\cdot e^{\gamma_\infty/\hbar}$ to obtain the density for $\Tr(f)$.
\end{enumerate}
It follows easily that this density for $Tr(f)$ satisfies a locality: at every point $x\in X$, it only depends on the Taylor expansion of $f$, the curvature of $X$ and the formal cohomology class $[\alpha]$ at $x$.


\subsection{One-loop exactness and the algebraic index}\label{section: computation-Feynman-graphs}
\

We present the explicit computation of $\Tr(1)$ as an example, whose formula is also known as the algebraic index theorem. First of all, a solution $\gamma_\infty$ of the QME \eqref{equation: quantum-master-equation} is obtained by running the homotopy group flow. This can be expressed as a Feynman graph expansion:\footnote{For some basic facts on Feynman graph expansion and Feynman weights, see Appendix \ref{section: Feynman graphs}. For more details and a proof of this lemma, see Costello's book \cite{Kevin-book}.}
\begin{lem}\label{lemma: gamma-infty-graph-expression}
	Let $\gamma\in\A_X^\bullet(\mathcal{W}_{X,\C})$ and $\gamma_\infty$ be defined as in Definition \ref{definition: gamma-infty}. Then $\gamma_\infty$ can be expressed as a sum of Feynman weights:
	\begin{equation*}\label{equation: gamma-infty-Feynman-weights}
	\gamma_\infty=\sum_{\mathcal{G}:\ \text{connected}}\frac{\hbar^{g(\mathcal{G})}}{|\text{Aut}(\mathcal{G})|}W_{\mathcal{G}}(P, d_{TX}\gamma),
	\end{equation*}
	where the sum is over all connected, stable graphs $\mathcal{G}$, and $g(\mathcal{G})$ denotes the genus of $\mathcal{G}$.
\end{lem}
It turns out that, in the K\"ahler case, if $\gamma$ is the solution of the Fedosov equation \eqref{equation: Fedosov-equation-gamma} obtained by quantizing Kapranov's $L_\infty$-algebra structure, then the Feynman graph expansion of the QME solution $\gamma_\infty$ involves {\em only trees and one-loop graphs}; in other words, $\gamma_\infty$ gives a {\em one-loop exact BV quantization} of the K\"ahler manifold $X$. This is in sharp contrast with the general symplectic case \cite{GLL}, in which the BV quantization involves all-loop quantum corrections. (To see this, only notice that the Fedosov connection in Fedosov's original construction is real, thus will have symmetries in $y^i$ and $\bar{y}^i$'s). 
The same kind of one-loop exactness was observed for the holomorphic Chern-Simons theory by Costello \cite{Kevin-CS} and for a sigma model from $S^1$ to the target $T^*Y$ (cotangent bundle of a smooth manifold $Y$) by Gwilliam-Grady \cite{Gwilliam-Grady}.

As a corollary, we obtain a succinct explicit expression of the algebraic index $\Tr(1) = \int_X\int_{Ber}e^{\tilde{R}_\nabla/2\hbar}e^{\gamma_\infty/\hbar}$ of the star product $\star_\alpha$. The latter is a cochain level enhancement of the result in \cite{GLL}: there the technique of equivariant localization was applied to show that all graphs of higher genera ($\geq 2$) give rise to exact differential forms after the Berezin integration and thus do not contribute after integration over $X$, while the Feynman weights associated to these graphs in our QME solution $\gamma_\infty$ vanish already on the cochain level.


\subsubsection{A weight on the BV bundle and one-loop exactness}
\

The key to one-loop exactness lies in the existence of a suitable weight on the BV bundle:
\begin{defn}
We define a {\em weight} on the BV bundle $\A_X^\bullet(\hat{\Omega}_{TX}^{-\bullet})[[\hbar]]$ by setting   
\begin{itemize}
 \item $|\bar{y}^i|=|d\bar{y}^i|=1$, and
 \item $|y^i|=|dy^i|=|dz^i|=|d\bar{z}^i|=|\hbar|=0$. 
\end{itemize}
\end{defn}
The following two lemmas are shown by simple computations.
\begin{lem}\label{lemma: gamma-weight}
For every closed $(1,1)$-form $\alpha$, let $D_{F,\alpha}=\nabla-\delta+\frac{1}{\hbar}[\gamma,-]_\star$ be the associated Fedosov connection obtained in Theorem \ref{proposition: Fedosov-connection-general}. Then in the $\hbar$ power expansion of $\gamma$:
$$
\gamma=\sum_{i\geq 0}\hbar^i\cdot\gamma_i,
$$
the weight of each term in $\gamma_0$ is either $0$ or $1$.  
\end{lem}
\begin{proof}
 The only term in $\gamma_0$ of weight $0$ is $\omega_{i\bar{j}}d\bar{z}^j\otimes \bar{y}^i$. All the other monomials in $\gamma_0$ have exactly one $\bar{y}^i$'s, and are thus of weight $1$. 
\end{proof}

\begin{lem}\label{lemma: operators-preserving-weight}
The following operators preserve the weight:
$$
\hbar\Delta, \hbar\partial_{P},\hbar\{-,-\}_\Delta.
$$
In particular, the homotopy group flow operator also preserves the weight. 
\end{lem}

Here is one of the main discoveries in this paper:
\begin{thm}[One-loop exactness]\label{theorem: gamma-infty-one-loop}
Let $\gamma$ be a solution of the Fedosov equation \eqref{equation: Fedosov-equation-gamma}.
Then 
$$W_{\mathcal{G}}(P, d_{TX}\gamma) = 0\quad \text{ whenever $b_1(\mathcal{G}) \geq 2$}.$$
As a result, the graph expansion of the QME solution $\gamma_\infty$, as defined in Definition \ref{definition: gamma-infty}, involves only trees and one-loop graphs:
$$
\gamma_\infty=\sum_{\mathcal{G}:\ \text{connected},\ b_1(\mathcal{G})=0,1}\frac{\hbar^{g(\mathcal{G})}}{|\text{Aut}(\mathcal{G})|}W_{\mathcal{G}}(P, d_{TX}\gamma). 
$$
\end{thm}
\begin{proof}
Let $\mathcal{G}$ be a connected graph with first Betti number $b_1(\mathcal{G})\geq 2$. Since every term in $\gamma$ has at most weight $1$, the total weights of those decorated vertices are bounded above by $|V(\mathcal{G})|$.
On the other hand, each propagator $P$ is of weight $-1$. Thus the internal edges of $\mathcal{G}$ decorated by $P$ has total weight $-(|V(\mathcal{G}|+b_1(\mathcal{G})-1)$. Hence we must have
$$
|V(\mathcal{G})|-(|V(\mathcal{G}|+b_1(\mathcal{G})-1)\geq 0,
$$
which implies that $b_1(\mathcal{G})\leq 1$. 
\end{proof}
\begin{rmk}
The argument here is similar to the proof of one-loop exactness of the holomorphic Chern-Simons theory by Costello \cite{Kevin-CS}. 
\end{rmk}

\subsubsection{A cochain level formula for the trace}
\

To derive a formula for the algebraic index $\Tr(1) = \int_X\int_{Ber}e^{\tilde{R}_\nabla/2\hbar}e^{\gamma_\infty/\hbar}$ of the star product $\star_\alpha$, we first extend the symbol map to the BV bundle 
 \begin{equation*}\label{equation: symbol-map-BV-bundle}
 \begin{aligned}
  \sigma: \A_X^\bullet(\hat{\Omega}_{TX}^{-\bullet})&\rightarrow\A_X^\bullet\\
  y^i, \bar{y}^j, dy^i,d\bar{y}^j&\mapsto 0. 
 \end{aligned}
 \end{equation*}

\begin{lem}\label{lemma: Berezin-integral-graph}
The Berezin integral $\int_{Ber}e^{\tilde{R}_\nabla/2\hbar}e^{\gamma_\infty/\hbar}$ can be expressed as follows:
$$
\int_X\int_{Ber}e^{\tilde{R}_\nabla/2\hbar}e^{\gamma_\infty/\hbar}=\int_X\sigma\left(e^{\hbar\iota_{\Pi}}(e^{\tilde{R}_\nabla/2\hbar}e^{\gamma_\infty/\hbar})\right)
$$
\end{lem}
\begin{proof}
A simple observation is that every term in $\tilde{R}_\nabla$ and $\gamma_\infty$ has the same degree in $dz^i,d\bar{z}^j$'s and $dy^i,d\bar{y}^j$'s. Since the integration $\int_X$ only takes the top degree term in $dz^i,d\bar{z}^j$'s, the only term that matters in $e^{\hbar\iota_{\Pi}}$ is $\frac{1}{n!}(\hbar\iota_{\Pi})^n$.
\end{proof}

Then we compute: 
\begin{align*}
e^{\hbar\iota_{\Pi}}(e^{\tilde{R}_\nabla/2\hbar}e^{\gamma_\infty/\hbar})
& = e^{\hbar\iota_{\Pi}}\left(e^{\tilde{R}_\nabla/2\hbar}\cdot \text{Mult}\int_{S^1[*]}e^{\hbar\partial_P+D}e^{\otimes\gamma/\hbar}\right)\\
& \overset{(*)}{=} e^{\hbar\iota_{\Pi}}\left(\text{Mult}\int_{S^1[*]}e^{\hbar\partial_P+D}(e^{\tilde{R}_\nabla/2\hbar}d\theta_1)\otimes e^{\otimes\gamma/\hbar}\right)\\
& = \text{Mult}\int_{S^1[*]}e^{\hbar(\iota_{\Pi}+\partial_P)}\circ e^{D}\left((e^{\tilde{R}_\nabla/2\hbar}d\theta_1)\otimes e^{\otimes\gamma/\hbar}\right)\\
& = \exp\left(\hbar^{-1}\cdot\sum_{\mathcal{G}\ \text{connected}}\frac{\hbar^{g(\mathcal{G})}}{|\text{Aut}(\mathcal{G})|}W(P+\iota_{\Pi}, d\theta\otimes (d_{TX}\gamma+\frac{1}{2}\tilde{R}_\nabla))\right);
\end{align*}
here the equality $(*)$ follows from the facts that $\partial_P$ cannot be applied to $\tilde{R}_\nabla/2$ by type reasons and that $D(\tilde{R}_\nabla)=0$. 

This computation says that the integrand whose integral gives rise to $\Tr(1)$ can be expressed as a sum of Feynman weights. But note that they are different from those for $\gamma_\infty$ because the propagator is now $\hbar(\partial_P+\iota_{\Pi})$. Similar to Theorem \ref{theorem: gamma-infty-one-loop}, this graph sum also involves only trees and one-loop graphs:
\begin{prop}\label{prop:one-loop-exact-algebraic-index}
In the Feynman graph expansion
$$
\sum_{\mathcal{G}\ \text{connected}}\frac{\hbar^{g(\mathcal{G})}}{|\text{Aut}(\mathcal{G})|}W(P+\iota_{\Pi}, d_{TX}\gamma+\frac{1}{2}\tilde{R}_\nabla),
$$
only terms which correspond to graphs with first Betti number $0$ and $1$ are non-vanishing.
\end{prop}

We now proceed to compute the Feynman weights associated to trees and one-loop graphs.
To clarify the computation, we first decompose the terms labeling the vertices and edges respectively. For the vertices, recall that the term $\gamma$ in the Fedosov connection has the $\hbar$-power expansion:
$$
\gamma=\sum_{i\geq 0}\hbar^i\gamma_i.
$$
For later computations, we give a detailed description of these $\gamma_i$'s:
\begin{itemize}
 \item For $\gamma_0$, we know that
$$
\gamma_0=\omega_{i\bar{j}}(dz^i\otimes\bar{y}^j-d\bar{z}^j\otimes y^i)+\gamma_0',
$$
where all terms of $\gamma_0'$ are at least cubic, with the leading term $R_{i\bar{j}k\bar{l}}d\bar{z}^j\otimes y^iy^k\bar{y}^l$. 
\item For every $i>0$, the leading term of $\gamma_i$ is given by $(\delta^{1,0})^{-1}(\alpha_i)$.
\end{itemize}
For the propagators, notice that $\partial_{P}$ and $\iota_{\Pi}$ are respectively tensor products of forms on $S^1[*]$ and tensors on the BV bundle of $X$. They correspond to the {\em analytic} and {\em combinatorial} parts of the propagators respectively.

We assign colors to both the vertices and edges, according to the previous decomposition of the functionals and propagators. For edges, we have
\begin{enumerate}
 \item a blue edge is labeled by $\partial_{P_1}$;
 \item a black edge is labeled by $\partial_{P_2}$;
 \item a red edge is labeled by the operator $\iota_{\pi}$;
 \item a yellow vertex is labeled by $\omega_{i\bar{j}}(dz^i\otimes d\bar{y}^j-d\bar{z}^j\otimes dy^i)$;
 \item a purple vertex is labeled by $d_{TX}\gamma_0'$;
 \item a green vertex is labeled by $\sum_{i>0}d_{TX}\gamma_i$;
 \item a a blue vertex is labeled by $\frac{1}{2}\tilde{R}_\nabla$.
\end{enumerate}
In a graph, when we do not distinguish $\partial_{P_1}$ or $\partial_{P_2}$, we assign black color to this edge. Moreover, since the analytic parts of $\partial_{P_1}, \partial_{P_2}$ and $\iota_{\Pi}$ are all given by contractions using the inverse of the K\"ahler form, we assign an orientation on each internal edge, going from the holomorphic derivatives to the anti-holomorphic derivatives.

Some sample pictures are listed here:
$$
\figbox{0.39}{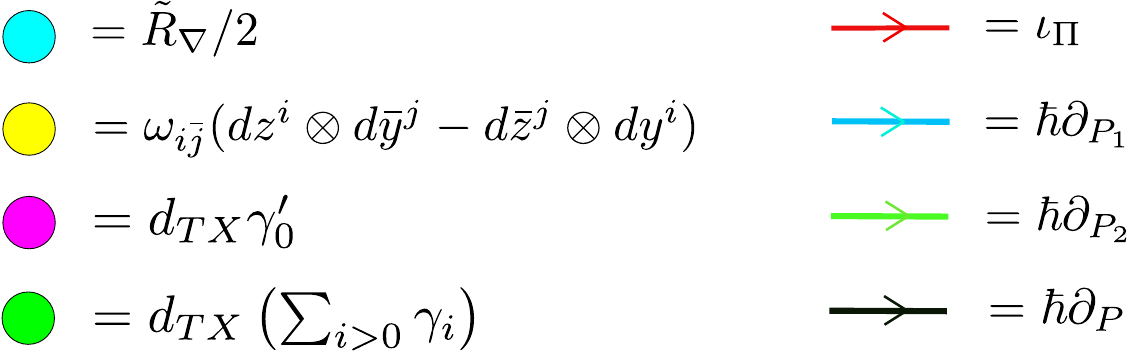}
$$

\begin{lem}\label{lemma: properties-vertex-edges}
The vertices in our Feynman graphs have the following properties:
\begin{enumerate}
 \item A purple vertex is at least trivalent with exactly one incoming tail and at least two outgoing tails. 
 \item The tails on a green vertex must be outgoing.
 \item A blue vertex has exactly one incoming tail and one outgoing tail.
 \item Every yellow, purple or green vertex can be connected to at most one red internal edge.
\end{enumerate}
\end{lem}
\begin{proof}
All the statements follow by considering the types of the sections of the BV bundle labeling the corresponding colored vertices. For instance, $(2)$ follows from the fact that the Weyl bundle component of $\gamma_i$ $(i>0)$ lives in $\mathcal{W}_X$.  
\end{proof}

In terms of Feynman weights, the procedure of taking the symbol map $\sigma$ in Lemma \ref{lemma: Berezin-integral-graph} corresponds to taking only those connected graphs with no tails (external edges) but only internal edges. Let us first consider trees.
\begin{prop}\label{proposition: tree-external-edges}
Suppose $\mathcal{G}$ is a tree which contains at least one vertex labeled by $\gamma_0'$ (i.e., a purple vertex). Then this tree has at least one outgoing tail (external edge) which can be contracted by $\partial_{P_1}$ or $\partial_{P_2}$ (i.e., blue or black). 

\end{prop}
\begin{proof}
We have seen that a purple vertex has to be at least trivalent. It is clear that every univalent vertex can only be connected by an internal edge labeled by $\iota_{\Pi}$. Let $v_i$ denote the number of vertices in $\mathcal{G}$ which are $i$-valent.  Let $\mathcal{G}'$ denote the graph obtained from $\mathcal{G}$ by deleting those blue and yellow vertices, and also red edges (both internal and external). In particular, every internal and external edge must be either blue or black, and the number of $n$-valent vertices in $\mathcal{G}'$ is $v_{n+1}$, for $n\geq 1$. 

Suppose $v_2=0$. Then a simple counting shows that there exists in $\mathcal{G}'$ at least one outgoing external edge. If $v_2>0$, we choose one of these univalent vertices in $\mathcal{G}'$, whose tail must be outgoing. We can draw a line starting from this vertex along the edges with directions and go as far as possible. Now we need the following simple fact: Suppose a vertex is at least bivalent in $\mathcal{G}$, then either all tails are outgoing, or at least one outgoing and exactly one incoming. This says that the line that we are drawing has to stop at an outgoing tail on a vertex which is at least bivalent.  
\end{proof}
 

As a corollary, we have a full list of trees {\em without} external edges:
\begin{cor}\label{cor:trees}
The only trees which survive under the symbol map are listed below, which exactly give rise to $\omega_\hbar$. 
$$
\omega_{\hbar}=\figbox{0.22}{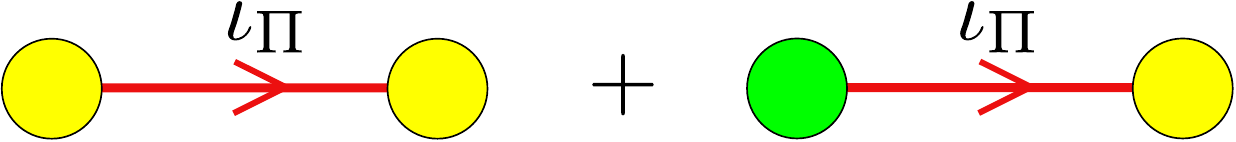}
$$
\end{cor}
\begin{proof}
According to Proposition \ref{proposition: tree-external-edges}, these trees cannot include a purple vertex.  We first show that they cannot include blue vertices either: A blue vertex has exactly one incoming tail and one outgoing tail. Since a green vertex has only outgoing tails, the outgoing tail of a blue vertex has to be connected to a yellow vertex, as shown in the following picture:
$$
\figbox{0.22}{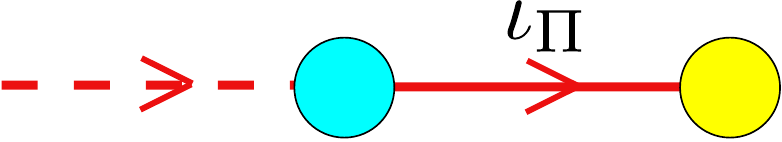}=\ \frac{\sqrt{-1}}{2}R_{i\bar{j}k\bar{l}}dz^i\wedge d\bar{z}^j\wedge dz^k\otimes d\bar{y}^l.
$$
But then this has to vanish since $R_{i\bar{j}k\bar{l}}$ is symmetric in $i$ and $k$. Therefore such a graph can only contain yellow and green vertices. These green vertices  must be univalent because there is no way to contract its blue or black tails. Hence the graphs satisfying the condition in this corollary have to be as stated. 
\end{proof}


Next we turn to one-loop graphs. Since every one-loop graph can be obtained by attaching trees to a {\em wheel} (see Definition \ref{definition: wheel}), we first look at all the possible wheels (possibly with tails). It is clear that there are three types of wheels, according to the labeling of the internal edges of the wheel:
\begin{enumerate}
 \item All the edges on the wheel are labeled by either $P_1$ or $P_2$.
 \item All the edges on the wheel are labeled by $\iota_{\Pi}$.
 \item The labelings of edges on the wheel include both $P$ (either $P_1$ or $P_2$) and $\iota_{\Pi}$.
\end{enumerate}

\begin{prop}\label{proposition: all-possible-wheels}
The symbol of a one-loop Feynman weight coming from a wheel of type $(3)$ vanishes.  
\end{prop}
\begin{proof}
Let $\mathcal{G}$ denote such a one-loop graph. A simple observation is that every vertex on the wheel of $\mathcal{G}$ is either labeled by $\gamma_0'$ or $\tilde{R}_\nabla$. We now choose a vertex $v$ such that the two edges on the wheel adjacent to $v$ are labeled by $\iota_{\Pi}$ and $P$ respectively, as shown in the following picture:
$$
\figbox{0.28}{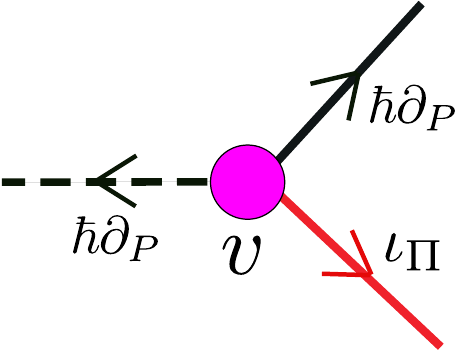}
$$
In particular, $v$ must be labeled by $\gamma_0'$, and is at least trivalent. Thus $v$ must have at least one outgoing tail (the dotted line in the above picture) which is {\em not} on the wheel. In order for the symbol of this Feynman integral to be non-vanishing, we have to attach a tree to this outgoing tail. Such a tree must include at least one vertex labeled by $\gamma_0'$, so that this vertex can be connected to the dotted line in the above picture. Thus by Proposition \ref{proposition: tree-external-edges}, this tree must also contain an outgoing tail. But this implies that the symbol of the Feynman integral of $\mathcal{G}'$ vanishes. 
\end{proof}

\begin{cor}
A one-loop graph which survives under the symbol map (equivalently, without external edges) must be of the following two types:
$$
\figbox{0.5}{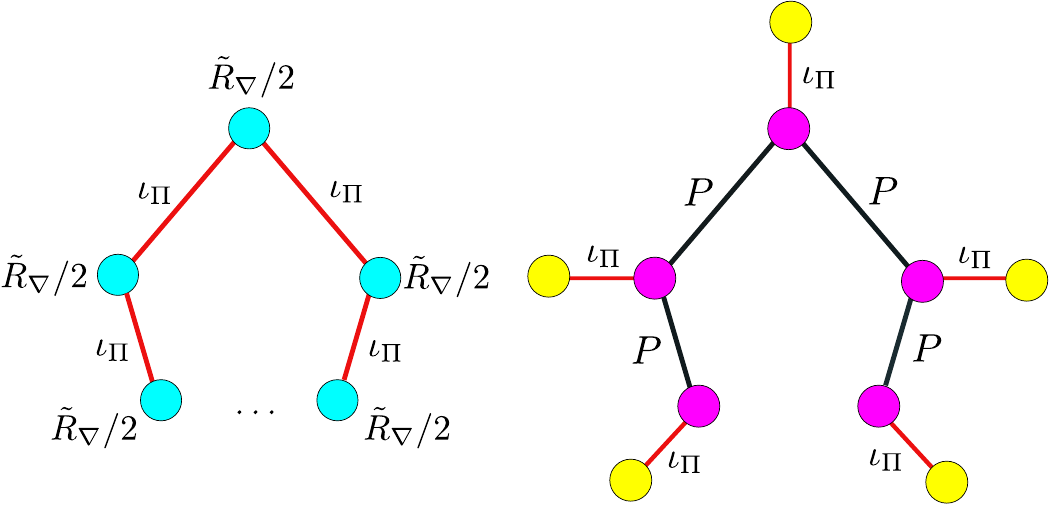}
$$
\end{cor}
\begin{proof}
Let $\mathcal{G}$ be such a one-loop graph. If $\mathcal{G}$ is of type $(2)$, then every vertex on the wheel must be labeled by $\tilde{R}_\nabla$, as shown on the left.
If $\mathcal{G}$ is of type $(1)$, then every vertex on the wheel must be labeled by $\gamma_0'$. A similar argument as in Proposition \ref{proposition: all-possible-wheels} shows that these vertices must be exactly trivalent, as shown on the right. 
\end{proof}

We can also see that if the edges of the wheel on such a one-loop graph are labeled by $\partial_P$, then they have to be of the same color:
\begin{prop}
If a wheel contains edges labeled by both $\partial_{P_1}$ and $\partial_{P_2}$, then the corresponding Feynman weights vanish. In other words, the following Feynman weights vanish:
$$
\figbox{0.22}{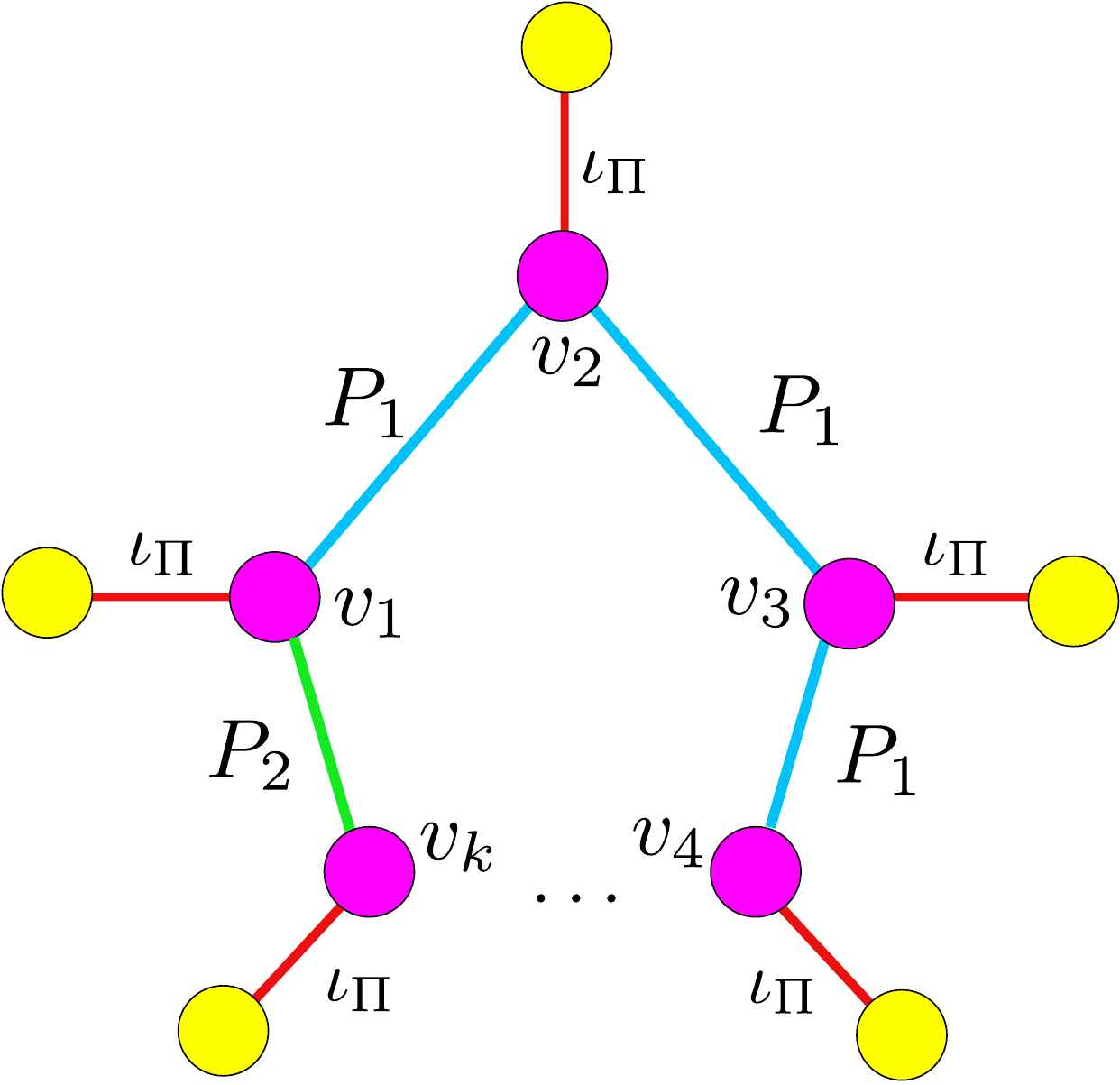}
$$
\end{prop}
\begin{proof}
We focus on the vertex $v_1$ in the above picture. Notice that both the red and green edges incident to $v_1$ only contribute constants to the analytic part of the propagator. The analytic part of the propagator $P_1$ labeling the blue edge between $v_1$ and $v_2$ is $u-1/2$ as in \eqref{definition: propagator-no-polarization}. Thus the Feynman weight of the above graph is a multiple of the following integral
$$
\int_{u=0}^1 (u-\frac{1}{2})du=0,
$$
which must vanish.
\end{proof}

To summarize, the only one-loop graphs which contribute non-trivially to the integrand whose integral gives rise to $\Tr(1)$ are of the following 4 types: 
\begin{equation}\label{picture:1_loop_graphs_1}
\figbox{0.58}{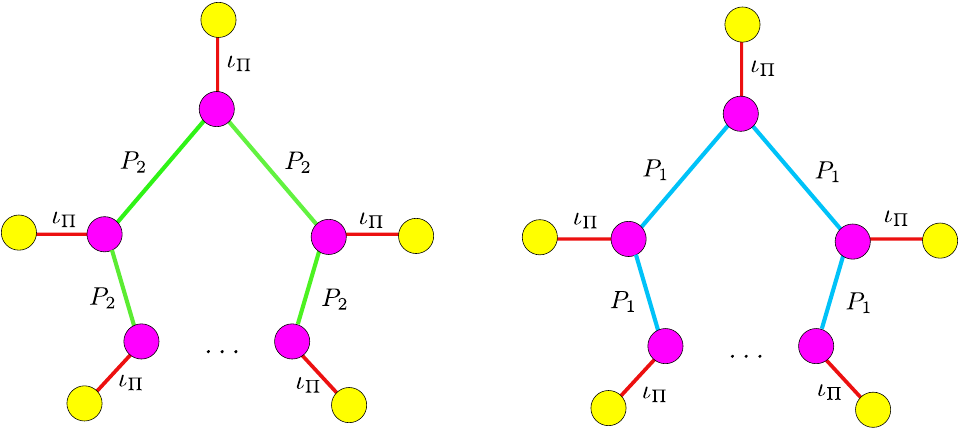}
\end{equation}
\begin{equation}\label{picture:1_loop_graphs_2}
\figbox{0.58}{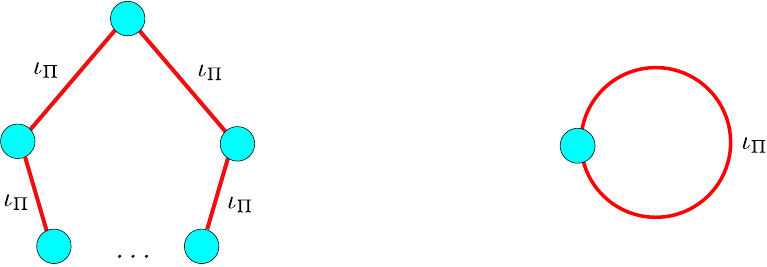}
\end{equation}

In order to have a more precise computation of the Feynman weights, we first compute the Feynman weights of the following two types of line graphs with $m\geq 2$ vertices ($m$ purple vertices in the left picture):
\begin{equation}\label{picture: two-vertices-cancellation}
\figbox{0.56}{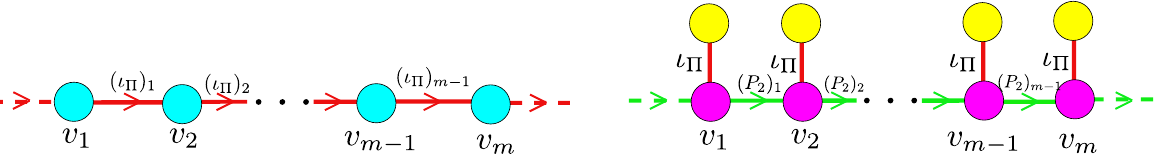}
\end{equation}
We label both the internal edges and vertices to avoid the issue of automorphisms of graphs and ordering of the contraction of propagators with vertices. 
\begin{lem}
The Feynman weight for the graph on the left of \eqref{picture: two-vertices-cancellation} is given by 
$$
-\frac{1}{2^m}\cdot (\omega^{k_1\bar{l}_2}\omega^{k_2\bar{l}_3}\cdots \omega^{k_{m-1}\bar{l}_m})(R_{i_1\bar{j}_1k_1\bar{l}_1}R_{i_2\bar{j}_2k_2\bar{l}_2}\cdots R_{i_m\bar{j}_mk_m\bar{l}_m})(dz^{i_1}d\bar{z}^{j_1}\cdots dz^{i_m}d\bar{z}^{j_m})\otimes (d\bar{y}^{l_1}\wedge dy^{k_m})
$$
while that for the graph on the right of \eqref{picture: two-vertices-cancellation} is given by 
$$
\frac{1}{2^{m-1}}\cdot (\omega^{k_1\bar{l}_2}\omega^{k_2\bar{l}_3}\cdots \omega^{k_{m-1}\bar{l}_m})(R_{i_1\bar{j}_1k_1\bar{l}_1}R_{i_2\bar{j}_2k_2\bar{l}_2}\cdots R_{i_m\bar{j}_mk_m\bar{l}_m})(dz^{i_1}d\bar{z}^{j_1}\cdots dz^{i_m}d\bar{z}^{j_m})\otimes (\bar{y}^{l_1}y^{k_m}).
$$
\end{lem}
\begin{proof}
When $m=2$, the Feynman weight of the left picture can be explicitly computed as follows: 
\begin{align*}
 &\omega^{p_1\bar{q}_1}\left(\iota_{\partial_{z^{p_1}}}(\tilde{R}_\nabla/2)\otimes\iota_{\partial_{\bar{z}^{q_1}}}(\tilde{R}_\nabla/2)\right)\\
 =&\ \frac{1}{4}\omega^{p_1\bar{q}_1}\cdot\text{Mult}\left(\iota_{\partial_{z^{p_1}}}(R_{i_1\bar{j}_1k_1\bar{l}_1}dz^{i_1}\wedge d\bar{z}^{j_1}\otimes dy^{k_1}\wedge d\bar{y}^{l_1})\otimes\iota_{\partial_{\bar{z}^{q_1}}}(R_{i_2\bar{j}_2k_2\bar{l}_2}dz^{i_2}\wedge d\bar{z}^{j_2}\otimes dy^{k_2}\wedge d\bar{y}^{l_2})\right)\\
 =&\ -\frac{1}{4}\omega^{k_1\bar{l}_2}\cdot\text{Mult}\left((R_{i_1\bar{j}_1k_1\bar{l}_1}dz^{i_1}\wedge d\bar{z}^{j_1}\otimes  d\bar{y}^{l_1})\otimes(R_{i_2\bar{j}_2k_2\bar{l}_2}dz^{i_2}\wedge d\bar{z}^{j_2}\otimes dy^{k_2})\right)\\
 =&\ -\frac{1}{4}\omega^{k_1\bar{l}_2}\cdot(R_{i_1\bar{j}_1k_1\bar{l}_1}R_{i_2\bar{j}_2k_2\bar{l}_2}dz^{i_1}\wedge d\bar{z}^{j_1}\wedge dz^{i_2}\wedge d\bar{z}^{j_2}\otimes(d\bar{y}^{l_1}\wedge dy^{k_2})).
\end{align*}
The statement for general $m>2$ follows by a simple induction. 

A simple computation shows that if the purple vertex is trivalent, then the Feynman weight of the following picture is $R_{i\bar{j}k\bar{l}}dz^i\wedge d\bar{z}^j\otimes y^k\bar{y}^l$:
$$
\figbox{0.5}{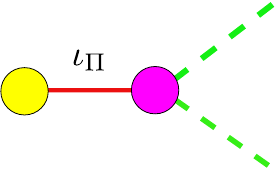}
$$
And the Feynman weight of the right picture of \eqref{picture: two-vertices-cancellation} follows from a straightforward computation which we omit here.

Thus we indeed get the cancellation. For wheels with more vertices, the cancellation can be proved by an induction on the number of vertices on the wheels. 
\end{proof}

A simple consequence of this computation is the following
\begin{prop}\label{prop:1-loop-graphs}
The first type (i.e., the left picture in \eqref{picture:1_loop_graphs_1}) and third type (i.e., the left picture in \eqref{picture:1_loop_graphs_2}) one-loop Feynman weights cancel with each other. 
\end{prop}
\begin{proof}
The first type and third type one-loop graphs are exactly obtained by connecting the starting and ending tails of the graphs shown in the picture \eqref{picture: two-vertices-cancellation}. It follows from the previous lemma that the Feynman weight for the first type is of the form
$$
-C\cdot \frac{1}{2^m}\cdot (\omega^{k_1\bar{l}_2}\omega^{k_2\bar{l}_3}\cdots \omega^{k_{m-1}\bar{l}_m}\omega^{k_m\bar{l}_1})(R_{i_1\bar{j}_1k_1\bar{l}_1}R_{i_2\bar{j}_2k_2\bar{l}_2}\cdots R_{i_m\bar{j}_mk_m\bar{l}_m})(dz^{i_1}d\bar{z}^{j_1}\cdots dz^{i_m}d\bar{z}^{j_m}),
$$
where the constant $C$ arises from the combinatorics of graphs, while that for the third type is of the form
$$
C\cdot \frac{1}{2^m}\cdot (\omega^{k_1\bar{l}_2}\omega^{k_2\bar{l}_3}\cdots \omega^{k_{m-1}\bar{l}_m}\omega^{k_m\bar{l}_1})(R_{i_1\bar{j}_1k_1\bar{l}_1}R_{i_2\bar{j}_2k_2\bar{l}_2}\cdots R_{i_m\bar{j}_mk_m\bar{l}_m})(dz^{i_1}d\bar{z}^{j_1}\cdots dz^{i_m}d\bar{z}^{j_m}).
$$
\end{proof}

The second type one-loop graphs, i.e., connected wheels with only blue edges on the wheel (the right picture in \eqref{picture:1_loop_graphs_1}), give rise precisely to the logarithm of $\hat{A}$ genus, as was shown by Grady and Gwilliam \cite{Gwilliam-Grady}:
\begin{prop}[Corollary 8.6 in \cite{Gwilliam-Grady}]\label{prop: one-loop-logarithm-A-roof-genus}
The Feynman weights corresponding to the following one-loop graphs give rise to the logarithm of the $\hat{A}$ genus of $X$.
$$
\figbox{0.58}{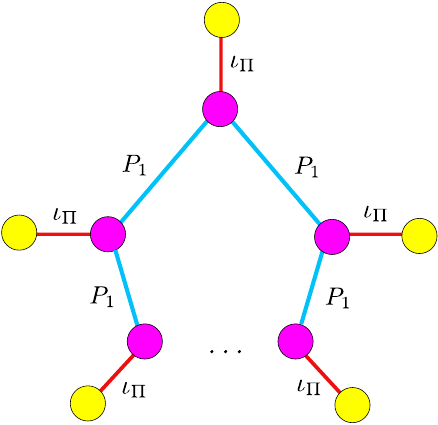}
$$
\end{prop}
Finally, the contribution from the fourth type one-loop graphs, i.e., a tadpole graph (the right picture in \eqref{picture:1_loop_graphs_2}), is described by the following lemma:
\begin{lem}\label{lemma: weight-tadple-graph}
The Feynman weight of the tadpole graph is given by $\frac{1}{2}\Tr(\mathcal{R}^+)$, i.e.,
$$
\frac{1}{2}\Tr(\mathcal{R}^+)=\figbox{0.66}{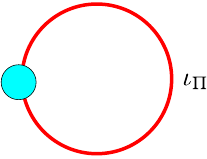},
$$
where $\mathcal{R}^+$ is given in \eqref{equation: R-plus}, \eqref{equation: trace-R-plus}.
\end{lem}
\begin{proof}
The Feynman weight of the tadpole graph is explicitly given by
\begin{align*}
 \hbar\iota_{\Pi}\left(\frac{1}{2\hbar}\tilde{R}_\nabla\right)& = \iota_{\Pi}\left(\frac{\sqrt{-1}}{2}R_{i\bar{j}k\bar{l}}dz^i\wedge d\bar{z}^j\otimes dy^k\wedge d\bar{y}^l\right)\\
 & = -\frac{\sqrt{-1}}{2}\omega^{k\bar{l}}R_{i\bar{j}k\bar{l}}dz^i\wedge d\bar{z}^j = \frac{1}{2}\Tr(\mathcal{R}^+).
\end{align*}
\end{proof}

Combining Proposition \ref{prop:one-loop-exact-algebraic-index}, Corollary \ref{cor:trees}, Proposition \ref{prop:1-loop-graphs}, Proposition \ref{prop: one-loop-logarithm-A-roof-genus} and Lemma \ref{lemma: weight-tadple-graph}, we arrive at our second main result, which is a cochain level formula for the algebraic index:
\begin{thm}\label{theorem:algebraic-index-theorem}
	Let $\gamma$ be a solution of the Fedosov equation \eqref{equation: Fedosov-equation-gamma} and $\gamma_\infty$ be the associated solution of the QME as defined in Definition \ref{definition: gamma-infty}. Then we have
$$\sigma\left(e^{\hbar\iota_{\Pi}} (e^{\tilde{R}_\nabla/2\hbar}e^{\gamma_\infty/\hbar}) \right)
= \hat{A}(X)\cdot e^{-\frac{\omega_\hbar}{\hbar}+\frac{1}{2}\Tr(\mathcal{R}^+)} = \text{Td}(X)\cdot e^{-\frac{\omega_\hbar}{\hbar}+\Tr(\mathcal{R}^+)},$$
where $Td(X)$ is the Todd class of $X$. 
\end{thm}
\begin{proof}
We only need to show the second equality, which follows from the formula 
$Td(X)=\hat{A}(X)\cdot e^{-\frac{1}{2}\Tr(\mathcal{R}^+)}$.
\end{proof}

Applying Lemma \ref{lemma: Berezin-integral-graph} gives the {\em algebraic index theorem}:
\begin{cor}\label{corollary:algebraic-index-theorem}
	The trace of the function $1$ is given by 
	\begin{align*}
	\Tr(1) = \int_X\hat{A}(X)\cdot e^{-\frac{\omega_\hbar}{\hbar}+\frac{1}{2}\Tr(\mathcal{R}^+)}
	=  \int_X\text{Td}(X)\cdot e^{-\frac{\omega_\hbar}{\hbar}+\Tr(\mathcal{R}^+)}.
	\end{align*}
\end{cor}

As we mentioned in the introduction, when $\alpha = \hbar\cdot\Tr(\mathcal{R}^+)$ or $\omega_\hbar=-\omega+\hbar\cdot\Tr(\mathcal{R}^+)$, we will prove in the forthcoming work \cite{CLL-PartIII} that the associated star product $\star_\alpha$ is exactly equal to the {\em Berezin-Toeplitz star product} studied in \cite{Bordemann-Meinrenken, Bordemann, Karabegov}. In this case, the algebraic index theorem is formulated as:
$$
\Tr(1)=\int_X\text{Td}(X)\cdot e^{\omega/\hbar}.
$$

\appendix
\section{Feynman graphs}\label{section: Feynman graphs}
In this section, we describe the basics of Feynman graphs. For more details, we refer the reader to \cite{Kevin-book}.
\begin{defn}\label{definition: Feynman-graphs}
A {\em graph} $\mathcal{G}$ consists of the following data:
\begin{enumerate}
 \item A finite set $V(\mathcal{G})$ of {\em vertices},
 \item A finite set $H(\mathcal{G})$ of {\em half edges},
 \item An involution $\sigma: H(G)\rightarrow H(G)$. The set of fixed points of this map is called the set of {\em tails} of $\mathcal{G}$, denoted by $T(\mathcal{G})$; a tail is also called an {\em external edge}.  The set of two-element orbits of this map is called the set of {\em internal edges} of $\mathcal{G}$, denoted by $E(\mathcal{G})$,
 \item A map $\pi: H(\mathcal{G})\rightarrow V(\mathcal{G})$ sending a half-edge to the vertex to which it is attached,
 \item A map $g: V(\mathcal{G})\rightarrow\Z_{\geq 0}$ assigning a {\em genus} to each vertex. 
\end{enumerate}
\end{defn}
\begin{rmk}
 In a picture of a graph, we will use solid lines and dotted lines to denote internal edges and tails respectively.
\end{rmk}

It is explained in \cites{Kevin-book, Manin} how to construct a topological space $|\mathcal{G}|$ associated to a graph $\mathcal{G}$, and we omit the details here. A graph $\mathcal{G}$ is called {\em connected} if $|\mathcal{G}|$ is connected. The {\em genus} is a graph is defined by
\begin{equation*}\label{equation: genus-graph}
 g(\mathcal{G}):=b_1(\mathcal{G})+\sum_{v\in V(\mathcal{G})}g(v).
\end{equation*}
Here $b_1(\mathcal{G})$ denote the first Betti number of $|\mathcal{G}|$. 
\begin{rmk}
We call a graph $\mathcal{G}$ a {\em tree} if $b_1(\mathcal{G})=0$, or a {\em one-loop graph} if $b_1(\mathcal{G})=1$. 
\end{rmk}
We also introduce the notion of a wheel graph:
\begin{defn}\label{definition: wheel}
A one-loop graph $\mathcal{G}$ is called a {\em wheel} if removing any of its internal edges will give rise to a tree, as in the following picture:
$$
\figbox{0.28}{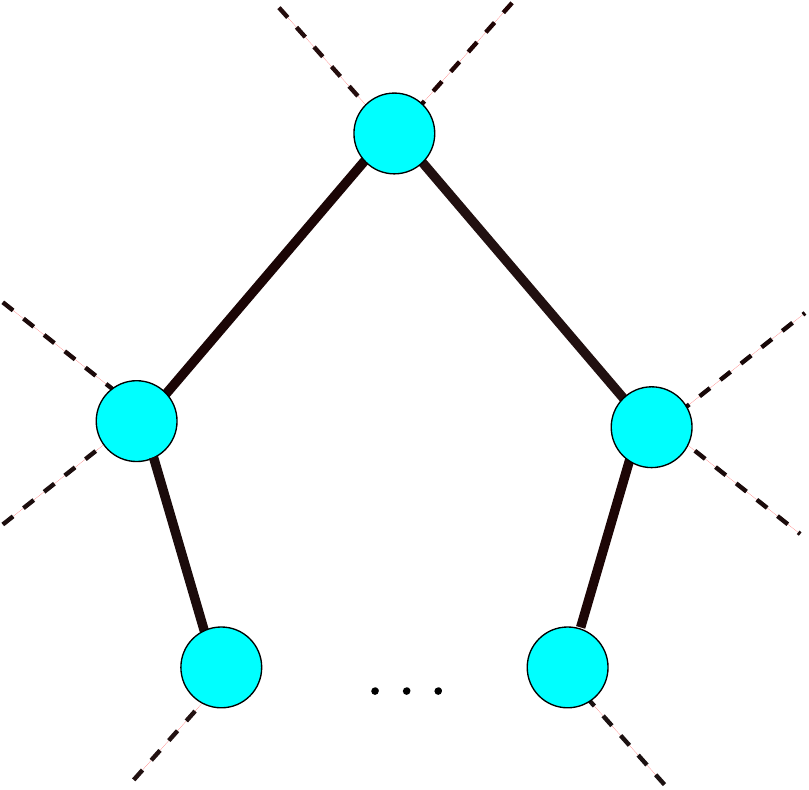}
$$
\end{defn}

Let $\E$ be a graded module over a base ring $R$, with $I\subset R$ a nilpotent ideal. Let $\E^*:=\Hom_R(\E,R)$ denote its $R$-linear dual (or continuous dual when there is a topology on $\E$). Let
$$
\hat{\mathcal{O}}(\E):=\prod_{k\geq 0}\Sym^k_R(\E^*),
$$
denote the space of formal functions on $\E$. 
\begin{eg}
We give the example in the K\"ahler quantization as an illustration: the base ring here consists of differential form $R=\A_X^*$ on $X$, where the nilpotent ideal is the $I=\A_X^{>0}$ is the subspace of forms of degree at least $1$. In particular, if we take the $R$-module to be $\E=\A_X^*\otimes_{\mathcal{C}^\infty_X}TX_\C$, then there is the canonical isomorphism 
$$
\hat{\mathcal{O}}(\E)\cong \A_X^*\otimes_{\mathcal{C}^\infty_X}\W_{X,\C}.
$$
\end{eg}

We define a subspace
$$
\mathcal{O}^+(\E)\subset\hat{\mathcal{O}}(\E)[[\hbar]]
$$
consisting of those formal functions which are at least cubic modulo $\hbar$ and the nilpotent ideal $\mathcal{I}$ in $R$. Let $F\in\mathcal{O}^+(\E)$ which we expand as
$$
F=\sum_{g,k\geq 0}F_g^{(k)},
$$
where $F_g^{(k)}:\E^{\otimes k}\rightarrow R$ is an $S_k$-invariant (continuous) map.

We will fix an element $P\in\Sym^2(\E)$ which we call the {\em propagator}. With $F$ and $P$, we will describe for every connected stable graph $\mathcal{G}$ the {\em Feynman weight}:
$$
W_{\mathcal{G}}(P,F)\in\mathcal{O}^+(\E).
$$
Explicitly, we label each vertex $v\in V(\mathcal{G})$ of genus $g(v)$ and valency $k$ by $F_{g(v)}^{(k)}$. This defines an element:
$$
F_v:\E^{\otimes H(v)}\rightarrow R,
$$
where $H(v)$ denotes the set of half-edges incident to $v$. We label each internal edge by the propagator 
$$
P_e=P\in\E^{H(e)},
$$
where $H(e)$ denotes the two half edges which together give rise to the internal edge $e$. We can then contract the tensor product of vectors in $E$ (from $E(\mathcal{G})$) with the tensor product of formal functions on $E$ (from $V(\mathcal{G})$) to yield an $R$-linear map:
$$
W_{\mathcal{G}}(P,F):\E^{T(\mathcal{G})}\rightarrow R. 
$$

\begin{defn}\label{definition: homotopic-group-flow}
We define the {\em homotopic renormalization group flow operator (HRG)} with respect to the propagator $P$
$$
W(P,-): \mathcal{O}^+(\E)\rightarrow\mathcal{O}^+(\E)
$$
by
$$
W(P,F):=\sum_{\mathcal{G}}\frac{\hbar^{g(\mathcal{G})}}{|\text{Aut}(\mathcal{G})|}W_{\mathcal{G}}(P,F),
$$
where the sum is over all connected graphs, and $\text{Aut}(\mathcal{G})$ denotes the automorphism group of $\mathcal{G}$. We can equivalently describe the HRG operator formally by the simple formula:
$$
e^{W(P,F)/\hbar}=e^{\hbar\partial_P}(e^{F/\hbar}),
$$
where $\partial_P$ denotes the second order differential operator on $\mathcal{O}(\E)$ given by contracting with $P$.  
\end{defn}



\end{document}